\documentclass[11pt]{article}
\usepackage{mathrsfs}  

\usepackage{authblk}

\usepackage{amsmath}
\usepackage{amsfonts}
\usepackage{amsthm}
\usepackage{eurosym}
\usepackage{graphicx}
\usepackage{amssymb}
\usepackage{epstopdf}
\usepackage{topcapt}
\usepackage{latexsym}
\usepackage{mathtools}
\usepackage{graphics}
\usepackage{subfigure}
\usepackage{dsfont}
\usepackage[english]{babel}
\usepackage{color}

\usepackage{float}

\newcommand{\ud}{\mathrm{d}}

\begin{document}

\long\def\TITLE#1{\hang{\bf#1}\vskip10pt}
\long\def\ADDRESS#1{\noindent {\rm [ #1 ]} \vskip10pt }
\long\def\ABSTRACT#1 {\parindent=0pt {\rm #1 }\parindent=1.5cm}
\long\def\AFFILIATION#1{\hang {\it #1} \vskip10pt}
\long\def\AUTHOR#1{\hang {\rm#1}}

\newtheorem{theorem}{Theorem}[section]
\newtheorem{algo}[theorem]{Algorithm}
\newtheorem{prop}[theorem]{Proposition}
\newtheorem{exempel}[theorem]{Example}
\newtheorem{remark}[theorem]{Remark}
\newtheorem{example}[theorem]{Example}
\newtheorem{conjecture}[theorem]{Conjecture}
\newtheorem{hypo}[theorem]{Hypothesis}
\newtheorem{details}[theorem]{Details}
\newtheorem*{details*}{Details}
\newtheorem{antag}[theorem]{Assumption}
\newtheorem{lemma}[theorem]{Lemma}
\newtheorem{definition}[theorem]{Definition}
\newtheorem{question}[theorem]{Question}
\newtheorem{reminder}[theorem]{Reminder}
\newtheorem*{reminder*}{Reminder}
\newtheorem{proposition}[theorem]{Proposition}
\renewcommand{\theequation}{\arabic{section}.\arabic{equation}}
\def\sqr#1#2{{\vcenter{\vbox{\hrule height.#2pt
\hbox{\vrule width.#2pt height#1pt \kern#1pt
\vrule width.#2pt}
\hrule height.#2pt}}}}
\font\largeb=cmb10 scaled\magstep2
\font\hugeb=cmb10 scaled\magstep3
\def\a{\alpha}
\def\Am{{A_{max}}}
\def\b{\beta}
\def\bR{{\bf R}}
\def\bb{{\cal b}}
\def\cB{{\cal B}}
\def\cC{{\cal C}}
\def\cD{{\cal D}}
\def\cE{{\cal E}}
\def\cF{{\cal F}}
\def\cG{{\cal G}}
\def\cO{{\cal O}}
\def\cL{{\cal L}}
\def\cS{{\cal S}}
\def\cX{{\cal X}}
\def\d{\partial}
\def\de{\delta}
\def\e{\varepsilon}
\def\f{\frac}
\def\g{\gamma}
\def\8{\infty}
\def\k{\kappa}
\def\l{\lambda}
\def\L{\Lambda}
\def\nn{\nonumber}
\def\oa{{\overline a_A}}
\def\ob{{\overline \beta}}
\def\ocF{{\overline{\cal F}}}
\def\of{{\overline f}}
\def\oF{{\overline F}}
\def\oh{{\overline h}}
\def\ome{\omega}
\def\Ome{\Omega}
\def\ops{{\overline \psi}}
\def\os{{\overline s}}
\def\oS{{\overline S}}
\def\oT{{\overline T}}
\def\ot{{\overline \tau}}
\def\ob{{\overline b}}
\def\obe{{\overline\beta}}
\def\obeta_2{{\overline \beta_2}}
\def\oh{{\overline h}}
\def\ob{{\overline \beta}}
\def\og{{\overline g}}
\def\oga{{\overline \gamma}}
\def\oi{{\overline i}}
\def\oI{{\overline I}}
\def\oK{{\overline K}}
\def\ok{{\overline k}}
\def\oL{{\overline L}}
\def\om{{\overline m}}
\def\omu{{\overline \mu}}
\def\on{{\overline n}}
\def\oo{{\overline o}}
\def\ophi{{\overline \phi}}
\def\oq{{\overline q}}
\def\oR{{\overline R}}
\def\oS{{\overline S}}
\def\ov{{\overline v}}
\def\ovi{{\overline \vi}}
\def\ow{{\overline w}}
\def\oW{{\overline W}}
\def\ox{{\overline x}}
\def\oX{{\overline X}}
\def\oy{{\overline y}}
\def\oY{{\overline Y}}
\def\oz{{\overline z}}
\def\ovarphi{{\overline \varphi}}
\def\ug{{\underline g}}
\def\uga{{\underline \gamma}}
\def\uta{{\underline \t}}
\def\0*{^{\odot *}}
\def\qed{\hfill$\sqr45$\bigskip}
\def\rM{{\rm M}}
\def\N{\mathds{N}}
\def\R{\mathds{R}}
\def\om{\omega}
\def\s{\sigma}
\def\Si{\Sigma}
\def\t{\tau}
\def\th{\theta}
\numberwithin{equation}{section}
\def\xa{{x_A}}
\def\vi{\varphi}
\def\xb{{x_b}}
\def\xm{{x_m}}
\def\Yx{Y^{\times}}
\def\yx{y^{\times}}
\def\tT{\tilde T}
\def\Proof{\noindent{\em Proof. }}
\def\<{\langle}
\def\>{\rangle}
%
%
\title{Stability results for a hierarchical size-structured population model \\ with  distributed delay}

\author[1]{Dandan Hu}
\author[2]{J\'{o}zsef Z. Farkas}
\author[1]{Gang Huang}
\affil[1]{School of Mathematics and Physics, China University of Geosciences, Wuhan, 430074, P. R. China}
\affil[2]{Departament de Matem\`{a}tiques 
 Universitat Aut\`{o}noma de Barcelona, Bellaterra, 08193,  Spain  }
\date{}                     
\setcounter{Maxaffil}{0}
\renewcommand\Affilfont{\itshape\small}


\date{}
\maketitle

\begin{abstract}
In this paper we investigate a structured population model with distributed delay. Our model incorporates two different types of nonlinearities. Specifically we assume that individual growth and mortality are affected by scramble competition, while fertility is affected by contest competition.  In particular, we assume that there is a hierarchical structure in the population, which affects mating success. The dynamical behavior of the model is analysed via linearisation by means of semigroup and spectral methods. In particular, we introduce a reproduction function and use it to derive linear stability criteria for our model.  Further we present numerical simulations to underpin the stability results we obtained. 

\end{abstract}

\section{Introduction}\label{sec:intro}

Population dynamics has been at the center of biomathematics since Malthus' exponential model of population growth. The renowned logistic equation is the classic example of a mathematical model for a self-regulating population. Simple models, like the logistic model are based on the premise that for example the per capita growth rate is solely determined by the total population size. Although these type of models tackle the issues of population self-regulation and stability, they fail to consider individual variations.  As a consequence, only population level processes can be accounted for, and their predictive power may be limited.

It is clear that an individual's activity in a specific population may be influenced not just by one-on-one interactions with other members of the population of the same physiological state, but also by interactions with individuals who are of a different state (e.g. older or younger, larger or smaller etc.) than themselves.
It has been proven that when population density rises, competition among individuals for a restricted resource increases, and individuals may compete for a variety of resources including food, space, shelter, and mates. It has been also shown that individuals of various species, such as fish, lizards, water buffalo, snails, and others, have a positive relationship between the quantity of available food and their own body size \cite{1984EEW}. A similar phenomenon occurs for  terrestrial plants, which rely on solar energy for photosynthesis. The survival of each plant is heavily influenced by the vertical component of plant size distribution, or the size related  hierarchy within a particular strand. Clearly, a taller plant is exposed to more light, and the energy is then channeled into their individual growth. Consequently, a hierarchical size-structured population model may prove to be useful to model such species, in particular when modelling intra-specific competition.  Indeed, hierarchical size-structured population models have been studied extensively in the literature. Without the desire of completeness we mention a few relevant recent papers, where the interested reader will also find further references \cite{FH2008,ASA2005,ASA2005-1,HZR2020,YDX2018}.

In the present paper we introduce and study a size-structured population model in which the birth rate is a function of an infinite dimensional interaction variable related to a hierarchy in the population (modelling contest competition), and the growth and death rates are functions of the total population size (modelling scramble competition). Hence our model incorporates two different types of nonlinearities. 
Specifically, we consider the following system that describes the dynamics of a hierarchical size-structured population model with delayed birth process.

\begin{equation}\label{1.1}
	\left\{\begin{array}{ll} 
		
		\frac{\partial p(s, t)}{\partial t}+\frac{\partial}{\partial s}(\gamma(s, P) p(s, t))+\mu(s, P) p(s, t)=0, \quad 0\leqslant s \leqslant m, \hspace{1mm} t>0,
		\\[0.8em]
		
		p(0, t)= \int_{0}^{m} \int_{-\theta}^{0} \beta(s, \tau, Q(s, t+\tau)) p(s, t+\tau) \ud \tau\ud s, \quad t > 0, 
		\\[0.8em] 
		
		p(s,\delta)=p^{0}(s,\delta), \quad  0<s\leqslant m, \hspace{1mm} \delta \in[-\theta, 0].
	\end{array}\right.
\end{equation}
Here $p(s,t)$ stands for the density of individuals with respect to size $s\in [0,m]$, where $m$ is the maximum size of an individual in the population. 
The functions $\gamma$ and $\mu$ denote individual growth and mortality rates respectively, which depend on the individual's own size $s$ as well as on the total population size
\begin{equation} \label{1.2}
	P(t)=\int_{0}^{m}p(s,t) \ud s.
\end{equation}
The function $\beta$ in Eqs. \eqref{1.1} stands for the fertility rate of an individual, which  depends on the size $s$ and a function of the population density (environment) specified as:
\begin{equation}\label{1.3}
	Q(s, t+\tau )=\alpha \int_{0}^{s} w(r) p(r, t+\tau ) \ud r+\int_{s}^{m} w(r) p(r, t+\tau ) \ud r, \quad s \in[0, m], t>0.
\end{equation}
The interaction variable $Q$ accounts for a hierarchy in the population impacting fertility/reproduction, where the parameter $ \alpha\,  (0\leq \alpha <1)$ determines the strength of the hierarchy between individuals of different sizes.  In particular,  $\alpha =0$ corresponds to an absolute hierarchical structure, in which large individuals in the population have an absolute advantage. The other limiting case  $\alpha =1$ describes a scenario with no hierarchical structure in the population, that is, each individual is in fair (scramble) competition when accessing resources. The parameter $\tau \in [-\theta,0], $ where $ \theta>0 $ is expressed as the maximum delay. The distributed delay through $Q$ is introduced here to account for the effect of delay through contest competition. Note the slightly unusual boundary condition we employ in our model \eqref{1.1}. From the physical point of view the flux of individuals at the minimal size is naturally $\gamma(0,P)p(0,t)$. Hence we tacitly assume that $\gamma(0,P)\equiv 1$, i.e. the growth rate is normalised such that newborns have the same growth speed independent of the standing population. This assumption yields a great deal of simplification in the linearisation and makes the computations much more tractable. In the rest of the paper we assume that the vital rates satisfy the following regularity assumptions:
$$
\begin{aligned}
	&\gamma=\gamma(s, P) \in C^{1}\left([0, m] ; C^{1}[0, \infty)\right),\, \gamma>0,\\
	&\mu=\mu(s, P) \in C\left([0, m] ; C^{1}[0, \infty)\right),\, \mu \geq 0,\\
	&\beta=\beta(s,\tau, Q) \in C\left([0, m]\times[-\theta, 0) ; C^{1}[0, \infty)\right),\,  \beta \geq 0,\\ 
	&w=w(s) \in C^{1}([0, m]),\,  w>0.
\end{aligned}
$$

Mathematical models of physiologically structured populations have been developed and investigated by numerous researchers over the past decades. Without completeness we mention here a few recent (and not so recent) works \cite{zth2022,JZ2005,JZ2008,F2007,D2007,OD2007,W1985,Fbn1984,hg2012,LY2009,FJ2016,jz2009,YJY2021,1998,hd2022}, where the interested reader will find further useful references. There are two main modelling approaches to build and study structured population models. The classic PDE modelling approach, which we employ here, utilises the natural density distribution of the population, and therefore the resulting models are typically formulated as first order hyperbolic equations with  non-local boundary conditions, such as the one we study here. For relatively simple PDE models one can often directly derive a renewal (integral) equation for the population birth rate, which is a delay equation. For more complicated models, in particular with infinite dimensional nonlinearities, such a direct approach is not necessarily convenient. In this case it is possible to build from basic biological principles a structured population model, which takes the form of a delay equation, or an even more abstract dynamical system. Then the question of equivalence between the two different  formulations naturally arises, which has been the subject of the recent papers: \cite{Farkas2016,Jozsef2021,Jozsef2023}.  For a linear model with distributed states at birth we studied in \cite{Farkas2016} the equivalence results we were able to establish are quite satisfying. However, for certain classes of nonlinear models the question of equivalence is much more complicated, and the delay equation formulation has an advantage in particular when studying qualitative properties via linearisation. Using the the framework of nonlinear semigroup theory it is possible to establish existence of solutions of the nonlinear PDE model on a suitable Banach space using the Crandall-Liggett theorem \cite{Crandall1971}, however the arising (solution) nonlinear semigroup cannot be shown to be continuously differentiable in general.   
This does not necessarily mean though, that using the formal linearisation we employed here stability results cannot be deduced in the PDE framework. In fact we expect that this is possible, hence we tacitly assumed that this is the case in this work. Indeed the specific examples presented in Section 7, also support this. 

To invoke the linearised stability principle from the delay formulation of a model, one can study the  equivalence of the two formulations, for example by means of a continuous map, which maps orbits of the PDE formulation to orbits of the delay formulation. For a different, size-structured predator-prey (usually referred to as a {\em Daphnia}) model, such equivalence between orbits was studied in the recent paper \cite{Jozsef2021}. Moreover, in the recent manuscript \cite{Jozsef2023}; the delay formulation of a single species hierarchical size-structured population model is studied. That model is equipped with the classical boundary condition, representing recruitment of newborn individuals, and it is assumed that mortality is constant; however the growth rate depends on an infinite dimensional environmental variable (due to the hierarchical structure), but not size explicitly. For this model we have verified directly that the characteristic equation deduced (in a similar fashion  as here) from the linearisation of the PDE model, is equivalent to the characteristic equation deduced from the linearisation of the delay formulation, which is a major indicator that linear stability results deduced in the PDE framework do indeed hold for the original nonlinear model. We also briefly discussed the equivalence between orbits in the two formulations, however due to the special delay formulation we employed (eliminating the infinite dimensional environmental variable) difficulties arise; please see Section 6 in the above mentioned manuscript for more details.

Researchers have been focusing on how to incorporate delays (e.g. maturation) in the recruitment process, in particular in the context of age-structured models, see e.g. \cite{SP2004,SP2005,yuanyuan2022,YDX2016,FXL2014}, where stability results were obtained using similar methods to the ones we deploy here.
We specifically mention the paper \cite{YDX2018}, in which a similar size-structured model was studied. However, in that model the growth and mortality rates only depend on size, and not on the total population size.
It is clear that in most populations, individual growth and survival are greatly correlated with the size of the standing population, and that if the total population size falls below a certain level the population will almost certainly die out, which is also known as Allee effect \cite{allee}. 
To account for this, in our model we incorporated growth and mortality rates, which do depend on the total population size,  which makes our model more realistic.

The main aim of our work is to introduce and study a hierarchical size-structured population model, which incorporates two significantly different types of nonlinearities and a delay in the recruitment process. We aim to demonstrate how to apply the theory of strongly continuous semigroups for this model. The rest of the paper is organised as follows. In Sect. 2, we first give conditions for the existence of a positive stationary solution of model \eqref{1.1} and formally linearise it around a steady state. We then recall some theoretical results, which we will utilise later when studying the linearisation of the model. In Sect. 3, we rewrite the linearised system as an abstract Cauchy problem, and then prove that it is governed by a strongly continuous semigroup of operators. In Sect. 4, we study important regularity properties of the governing linear semigroup. In Sect. 5, we derive an explicit characteristic equation characterising the point spectrum of the generator of the governing linear semigroup. In Sect. 6, we establish criteria for the linear stability and instability of steady states of our model. Finally, in Sect. 7, some examples will be presented and using numerical simulations we verify that the linear stability results obtained in the previous section are indeed valid for the original nonlinear model.

\section{Preliminaries}

It is clear that our model \eqref{1.1} admits the trivial stationary solution. We now establish necessary conditions for the existence of a positive stationary solution. Clearly, any non-trivial stationary solution $p_{*}(s)$ of \eqref{1.1} satisfies the following equations
\begin{equation}\label{2.6}
	\frac{\partial}{\partial s}\left[\gamma(s,P_{*}) p_{*}(s)\right]=-\mu(s,P_{*}) p_{*}(s),
\end{equation}
\begin{equation}\label{2.7}
	p_{*}(0)=\int_{0}^{m} \int_{-\theta}^{0} \beta\left(s, \tau, Q_{*}(s)\right) p_{*}(s) \d \tau \d s.
\end{equation}
The general solution of Eq. \eqref{2.6} is found as
\begin{equation}\label{2.8}
	p_{*}(s)=p_{*}(0) e^{-\int_{0}^{s} \frac{\mu(y,P_{*})+\gamma_{s}(y,P_{*})}{\gamma(y,P_{*})} \d y}.
\end{equation}
Substituting Eq. \eqref{2.8} into Eq. \eqref{2.7}, we observe that 
\begin{equation}\label{2.2}
	1=\int_{0}^{m} 
	\int_{-\theta}^{0} \beta\left(s, \tau, Q_{*}(s)\right) e^{-\int_{0}^{s} \frac{\mu(y,P_{*})+\gamma_{s}(y,P_{*})}{\gamma(y,P_{*})} \d y}  \d \tau \d s,
\end{equation}
if $p_{*}(0)\neq0$. Hence for $P \geq0$ and $ Q \geq0$, we define the basic reproduction function as
\begin{equation}\label{2.3}
	\mathscr{R}(P,Q)=\int_{0}^{m}  \int_{-\theta}^{0} \Pi(s,P) \beta(s, \tau, Q(s, t+\tau)) \d \tau \d s,
\end{equation}
where the function $\Pi$ is given for $0 \leq s \leq m$ by
\begin{equation}\label{2.4}
	\Pi(s,P)=e^{-\int_{0}^{s} \frac{\mu(y,P)+\gamma_{s}(y,P)}{\gamma(y,P)} \d y}. 
\end{equation}
By integration of Eq. \eqref{2.8}, we obtain
\begin{equation}\label{2.9}
	p_{*}(0)=\frac{P_{*}}{\int_{0}^{m} e^{-\int_{0}^{s} \frac{\mu(y,P_{*})+\gamma_{s}(y,P_{*})}{\gamma(y,P_{*})} \d y} \d s},
\end{equation}
where $P_{*}=\int_{0}^{m} p_{*}(s) \d s$ represents the positive population size at the steady state.
Finally, using Eq. \eqref{2.9} in Eq. \eqref{2.8}, we get 
\begin{equation}\label{2.5}
	p_{*}(s)=\frac{P_{*} \Pi(s,P_{*})}{\int_{0}^{m} \Pi(s,P_{*}) \d s}.
\end{equation}
Then the function $Q_{*}$, defined by
\begin{equation}\label{2.1}
	Q_{*}(s)=\alpha \int_{0}^{s} w(r) p_{*}(r) \d r+\int_{s}^{m} w(r) p_{*}(r) \d r,
\end{equation}
satisfies the equation
\begin{equation}
	\mathscr{R}\left(P_{*},Q_{*}\right)=1.
\end{equation}
We give the following existence result for the stationary solution of model \eqref{1.1}.

\begin{proposition}
	If $p_{*}(s)$ is a positive stationary solution of model \eqref{1.1}, then $p_{*}$ is defined by \eqref{2.5} and the function $Q_{*}$ satisfies \eqref{2.1} and $\mathscr{R}\left(P_{*},Q_{*}\right)=1$.
	
\end{proposition}

Given a stationary solution $p_{*}$, we linearise our model \eqref{1.1} by introducing the infinitesimal perturbation $u=u(s, t)$ and making the ansatz $p=u+p_{*}$. After inserting this expression into \eqref{1.1} and
omitting all nonlinear terms, we obtain the linearised problem
\begin{equation}{\label{2-2-linearized}}
\begin{aligned}
	0= & \frac{\partial}{\partial t} u(s, t)+\gamma_{*}(s)\frac{\partial}{\partial s} u(s, t) +\nu_{*}(s) u(s, t) + \varepsilon_{*}(s)U(t),
	\\ 
	u(0, t)=  & \int_{0}^{m} \int_{-\theta}^{0} \beta_{Q}(s, \tau, Q_{*})p_{*}(s) H(s, t+\tau) \ud \tau \ud s  \\
	&  +\int_{0}^{m} \int_{-\theta}^{0} \beta(s, \tau, Q_{*}) u(s, t+\tau) \ud \tau \ud s,
	\\   
	H(s,t)= & \alpha \int_{0}^{s} w(r) u(r,t) \ud r+\int_{s}^{m} w(r) u(r,t)\ud r, 
	\\ 
	u(s, \delta)= & u^{0}(s, \delta), H(\delta)=H^{0}(\delta), \quad \delta \in[-\theta, 0], 
\end{aligned}
\end{equation}
where we have set $H(s,t)= Q(s,t)-Q_{*}(s),\, 
U(t)=P(t)-P_{*}=\int_{0}^{m} u(s, t) \d s,$  and
$$
\begin{aligned}
	\gamma_{*}(s)&=\gamma(s,P_{*}),\\
	\nu_{*}(s)&=\gamma_{s}(s, P_{*})+\mu(s, P_{*}),\\
	\varepsilon_{*}(s)&=p_{*}(s) \left(\mu_{P}(s,P_{*})+\gamma_{sP}(s,P_{*})\right)+p_{*}^{\prime}(s) \gamma_{P}(s,P_{*}).
\end{aligned}
$$

We now recall some important definitions and results from the theory of linear operators, which we are going to utilise later on. First let us recall the following characterisation theorem due to Hille and Yosida, see e.g. \cite{AP1983}. 

\begin{lemma}  \label{theorem2.2}
	A linear operator $A$ is the infinitesimal generator of a $C_{0}$-semigroup of contractions $T(t)$, $t \geq 0$ if and only if \\
	(i) $A$ is closed and $\overline{D(A)}=X$;\\
	(ii) The resolvent set $\rho(A)$ of $A$ contains $\mathbb{R}^{+}$ and for every $\lambda>0$
	\begin{equation}
		\left\|R(\lambda:A)\right\| \leq \frac{1}{\lambda}.
	\end{equation}
\end{lemma}

In fact, Lemma \ref{theorem2.2} implies that any Hille-Yosida operator gives rise to a $C_{0}$-semigroup on the closure of its domain.

\begin{definition}\label{defination2.3}
	Let \((A,D(A))\) be  a linear operator on the Banach space $X$ and set
	\begin{align*}
		&X_0 = (\overline{D(A)}, \left\| \cdot\right\|) ; \\
		&A_{0} x=A x, \text { for } x \in D\left(A_{0}\right)=\left\{x \in D(A): A x \in X_{0}\right\}.
	\end{align*}
	Then the operator \(\left(A_{0}, D\left(A_{0}\right)\right)\) is called the part of \(A\) in \(X_{0}\).
\end{definition}

Particularly, if \((A, D(A))\) is a Hille-Yosida operator, its part \(\left(A_{0}, D\left(A_{0}\right)\right)\) generates a strongly continuous semigroup \(\left(T_{0}(t)\right)_{t \geq 0}\) on \(X_{0}\) (see e.g. \cite{AP1983}).

\begin{lemma}(see e.g. \cite{KJ2000}) \label{lemma2.4}
	Let the operator $A$ be a Hille-Yosida operator on a Banach space $X$. If the operator $B$ is a bounded linear operator on $X$, then the operator $A+B$ is also a Hille-Yosida operator on the Banach space $X$.
\end{lemma}

\begin{definition}
	Let $(A, D(A))$ be  a closed linear operator on a Banach space $X$. Then the point spectrum of $A$, denoted by $\sigma_{p}(A)$, is defined as
	$$
	\sigma_{p}(A):=\{\lambda \in \mathbb{C}: \lambda I-A : D(A) \rightarrow X\text { is not injective}\},
	$$
	and a crucial quantity $s(A)$, called the spectral bound of $A$, is denoted by
	$$
	s(A):=\sup \{\operatorname{Re} \lambda: \lambda \in \sigma(A)\}.
	$$
\end{definition}

\begin{definition}
	If $(A, D(A))$ is a generator of a $C_{0}$-semigroup $(T(t))_{t \geq 0}$, we 
	denote by $\omega_{0}(A)$ the growth bound of the semigroup $(T(t))_{t \geq 0}$, and define it as
	$$
	\omega_{0}(A):=\lim _{t \rightarrow+\infty} t^{-1} \log \|T(t)\|.
	$$
\end{definition}

The following lemmas 
(see e.g. \cite{KJ2000} and \cite{SL1993}) will be used to establish the positivity of the governing $C_{0}$-semigroup \((T(t))_{t \geq 0}\).

\begin{lemma}(Riesz–Schauder theory) \label{lemma2.8}
	Let $(A,D(A))$  be a compact operator on the Banach space $X$, then\\
	(i) $ 0 \in \sigma(A)$ when dim $ X=\infty$;\\
	(ii) $\sigma(A) \backslash \{0\}$ =  $\sigma_{p}(A)  \backslash \{0\}$;\\
	(iii) $\sigma(A)$ is a discrete set having no limit points except  0.
\end{lemma}

\begin{lemma}\label{le2.7} A strongly continuous semigroup \((T(t))_{t \geq 0}\) on a Banach lattice \(X\) is positive if and only if the resolvent \(R(\lambda, A)\) of its generator $A$ is positive for all sufficiently large \(\lambda\).
\end{lemma}

\section{Existence of a $C_{0}$-semigroup governing the linearised system}

In this section, to prove the well-posedness of the linearised problem \eqref{2-2-linearized}, 
we set up a $C_{0}$-semigroup framework on a suitable Banach lattice. For any positive stationary solution $p_{*}(s)$, we denote the Banach space 
$$\mathfrak{X}=L^{1}([0, m])$$
with the usual norm  \( \|\cdot\| \) and on this space we introduce the following operators
\begin{equation*} 
	(\mathfrak{A}_{m}f)(s)= -\gamma_{*}(s)f^{\prime}(s) - \nu_{*}(s)f(s) \text{ for } s\in[0,m]
\end{equation*}
with domain  \(D\left(\mathfrak{A}_{m}\right)=W^{1,1}(0, m) \),
\begin{equation*}
	(\mathfrak{B}_{m}f)(s)= -\varepsilon_{*}(s) \int_{0}^{m} f(s) \d s  \text{ for } s\in[0,m]
\end{equation*}
with domain \(D\left(\mathfrak{B}_{m}\right)=L^{1}(0, m)\).
The $m$ subscript indicates that the operators are specified on their maximum domain.
Moreover, we define the boundary operator
$$
\mathcal{P}: D\left(\mathfrak{A}_{m}\right) \rightarrow \mathbb{C}, \hspace{1mm} \mathcal{P}(f):=f(0),
$$
which  is used to express the boundary condition (\cite{SP2004}). Next, we introduce the delay operator 
$$
\begin{aligned}
	\Phi(y)= &\int_{0}^{m} \int_{-\theta}^{0} \beta(s, \tau, Q_{*}) y(s, \tau) \ud \tau \ud s\\
	&+ \int_{0}^{m} \int_{-\theta}^{0} p_{*}(s) \beta_{Q}(s, \tau, Q_{*}) \left(\alpha \int_{0}^{s} w(r) y(r,\tau) \d r+\int_{s}^{m} w(r) y(r,\tau)\d r \right) \ud \tau \ud s ,
\end{aligned}
$$
where \(y \in E=L^{1}([-\theta, 0], \mathfrak{X} ) \cong L^{1}((0,m) \times [-\theta, 0])\).
Then with these notations Eqs. \eqref{2-2-linearized} can be cast in the form of an abstract boundary delay system:
\begin{equation}\label{cauchy}
	\left\{\begin{array}{rl} 
		\frac{\ud}{\ud t}
		u(t) 
		= & \left(\mathfrak{A}_{m}+\mathfrak{B}_{m}\right)
		u(t), \quad t \geqslant 0 ,
		\\[0.8em]
		\mathcal{P}
		u(t)=& \Phi\left(
		u_{t} \right), 
		\\[0.8em] 
		u_{0}(t)=&
		u^{0}(t), \quad t \in[-\theta, 0],
	\end{array}\right.
\end{equation}
where $u^{0}(t):=u^{0}(\cdot, t)$,  \(u:[0,+\infty) \rightarrow L^{1}(0, m)\) is defined as \(u(t) := u(\cdot, t)\), and \(u_{t}:[-\theta, 0] \rightarrow\)
\(L^{1}(0, m)\) is the history segment defined in the usual way as
$$
u_{t}(\tau) :=  u(t+\tau), \quad \tau \in[-\theta, 0].
$$
In order to transform \eqref{cauchy} into an abstract Cauchy problem, on the Banach space \(E\), we introduce the differential operator
$$
\left(Y_{m}y\right)(\tau) := \frac{\ud}{\ud \tau}
y(\tau)
$$
with domain \(D\left(Y_{m}\right)=W^{1,1}([-\theta, 0],\mathfrak{X}).\) Moreover, we define  another boundary operator \(G: D\left(Y_{m}\right) \rightarrow \mathfrak{X}\)  as
$$
Gy:= y(0).
$$

\noindent Next, we consider the product space \(\mathscr{X}:=E \times \mathfrak{X}\), on which we define the matrix operator
$$
\mathscr{A}=\mathscr{A}_{1}+\mathscr{A}_{2},
$$
where
\begin{equation}\label{a1a2}
	\mathscr{A}_{1}:=\left(\begin{array}{cc}
		\hspace{0.2em} Y_{m} & 0 \\
		\hspace{0.2em} 0 & \hspace{0.2em} \mathfrak{A}_{m}
	\end{array}\right), \hspace{2mm} \mathscr{A}_{2}:=\left(\begin{array}{cc}
		\hspace{0.2em}0 & 0 \\
		\hspace{0.2em}0 & \hspace{0.3em} \mathfrak{B}_{m}
	\end{array}\right)
\end{equation}
with domain
\begin{align*}
	\begin{array}{rl}
		D(\mathscr{A}) & =D\left(\mathscr{A}_{1}\right) \\[1ex]
		& =\left\{ \hspace{0.2em}
		\left(\begin{array}{c}
			y \\
			f
		\end{array}\right)
		\in D\left(Y_{m}\right) \times D\left(\mathfrak{A}_{m}\right):
		\begin{array}{l}
			Gy=f \\
			\mathcal{P}f=\Phi y
		\end{array}\right\}.
	\end{array}
\end{align*}
We get the following abstract Cauchy problem 
\begin{equation}\label{uu} 
	 \left\{\begin{array}{l}
		\mathscr{U}^{\prime}(t)=\mathscr{A} \mathscr{U}(t), \quad t \geqslant 0, \\ 
		\mathscr{U}(0)=\mathscr{U}_{0},
	\end{array}\right.
\end{equation}
which corresponds to the operator\((\mathscr{A},D(\mathscr{A}))\) on the space \(\mathscr{X}\). Here $
\mathscr{U}(t)=\left(\begin{array}{c}
	
	u_{t}\\
	
	u(t) 
\end{array}\right)
$ denotes the function \(\mathscr{U}:[0,+\infty) \rightarrow \mathscr{X}\).

To establish the well-posedness of the abstract Cauchy problem \eqref{uu}, we will show that $ (\mathscr{A},D(\mathscr{A}))$ generates a $C_{0}$-semigroup on \(\mathscr{X}\).

First of all, we consider the Banach space \(\mathcal{X}:=E \times \mathfrak{X} \times \mathfrak{X} \times \mathbb{C}\) and the matrix operator
$$
\mathcal{A}:=\left(\begin{array}{cc|cc}
	Y_{m} \hspace{0.1em} & 0 \hspace{0.5em} & 0 & \hspace{0.1em} 0 \\
	-G \hspace{0.1em} & 0 \hspace{0.5em} & Id & \hspace{0.1em} 0 \\
	\hline 0 \hspace{0.1em} & 0 \hspace{0.5em} & \hspace{0.1em} \mathfrak{A}_{m} &  \hspace{0.1em} 0 \\
	\Phi \hspace{0.1em} & 0 \hspace{0.5em} & \hspace{0.1em} -\mathcal{P} & \hspace{0.1em}  0
\end{array}\right)
$$
with domain \(D(\mathcal{A})=D\left(Y_{m}\right) \times\{0\} \times D\left(\mathfrak{A}_{m}\right) \times\{0\}\).

\begin{proposition}\label{prop3.2}
	The operator \((\mathcal{A}, D(\mathcal{A}))\) is a Hille-Yosida operator on the Banach space \(\mathscr{X}\).
\end{proposition}
\begin{proof}
	The operator \(\mathcal{A}\) can be written as the sum of two operators on \(\mathcal{X}\) as \(\mathcal{A}=\) \(\mathcal{A}_{1}+\mathcal{A}_{2}\), where
	$$
	\mathcal{A}_{1}=\left(\begin{array}{cc|cc}
		Y_{m} \hspace{0.1em} & 0 \hspace{0.5em} & 0 & \hspace{0.1em} 0 \\
		-G \hspace{0.1em} & 0 \hspace{0.5em} & 0 & \hspace{0.1em} 0 \\
		\hline 0 \hspace{0.1em} & 0 \hspace{0.5em} & \hspace{0.1em} \mathfrak{A}_{m} &  \hspace{0.1em} 0 \\
		0 \hspace{0.1em} & 0 \hspace{0.5em} & \hspace{0.1em} -\mathcal{P} & \hspace{0.1em}  0
	\end{array}\right), \quad
	\mathcal{A}_{2}=\left(\begin{array}{cc|cc}
		0 \hspace{0.4em} & 0 \hspace{0.3em} & \hspace{0.1em} 0 & \hspace{0.2em} 0 \\
		0 \hspace{0.4em} & 0 \hspace{0.3em} & \hspace{0.15em} I d & \hspace{0.2em} 0 \\
		\hline 0 \hspace{0.4em} & 0 \hspace{0.3em} & \hspace{0.1em} 0 &  \hspace{0.2em} 0 \\
		\Phi \hspace{0.4em} & 0 \hspace{0.3em} &\hspace{0.1em} 0 &  \hspace{0.2em} 0
	\end{array}\right)
	$$
	with \(D\left(\mathcal{A}_{1}\right)=D(\mathcal{A})\) and \(D\left(\mathcal{A}_{2}\right)=\mathcal{X} .\)

	It is easy to see that the restriction \(\left(Y_{0}, D\left(Y_{0}\right)\right)\) of \(Y_{m}\) to the kernel of $G$ generates the nilpotent left shift semigroup \(\left(\mathfrak{T}_{0}(t)\right)_{t \geqslant 0}\) on \(E\) which is given by
	$$
	(\mathfrak{T}_{0}(t)y)(s,\tau)=\left\{\begin{array}{ll}
		y(s, t+\tau), & \text { if } t+\tau \leqslant 0, \\
		0, & \text { if } t+\tau>0.
	\end{array}\right.
	$$
	Similarly, one can verify by direct computations that
	the restriction 
	\(\left(\mathfrak{A}_{0}, D\left(\mathfrak{A}_{0}\right)\right)\) of \(\mathfrak{A}_{m}\) to the kernel of \(\mathcal{P}\)
	generates the positive semigroup \(\left(\Lambda_{0}(t)\right)_{t \geqslant 0}\) on \(\mathfrak{X}\) defined as
	$$
	\begin{array}{l}
		(\Lambda_{0}(t)f)(s)
		=\left\{\begin{array}{c}
			e^{-\int_{\Gamma^{-1}(\Gamma(s)-t)}^{s} \frac{\nu_{*}(y)}{\gamma_{*}(y)} \d y} f\left(\Gamma^{-1}(\Gamma(s)-t)\right) ,
			\text { if } t \leq \Gamma(s), \\[1ex]
			0, \text { if } t>\Gamma(s),
		\end{array}\right.
	\end{array}
	$$
	where
	\begin{equation}\label{gamma}
		\Gamma(s)=\int_{0}^{s} \frac{1}{\gamma_{*}(y)} \ud y .  
	\end{equation}
	
	Next we demonstrate that  \(\mathcal{A}_{1}\) is a Hille-Yosida operator. To this end note that  for any \(\lambda \in \mathbb{C} \text{ and } \tilde{f} \neq 0\), the resolvent equation
	$$
	\left(\lambda I-\mathfrak{A}_{0}\right)f= \tilde{f}
	$$
	has the implicit solution
	\begin{equation}\label{fg} 
		f(s)=e^{-\int_{0}^{s} \frac{\lambda+\nu_{*}(y)}{\gamma_{*}(y)} \ud y} \int_{0}^{s} \frac{\tilde{f}(\alpha)}{\gamma_{*}(\alpha)} e^{\int_{s}^{\alpha} \frac{\lambda+\nu_{*}(a)}{\gamma_{*}(a)} \ud a} \ud \alpha. 
	\end{equation}
	It shows that \(\sigma\left(\mathfrak{A}_{0}\right)= \emptyset \text{ as } \gamma_{*}(s)>0.\)
	In the same way,  \(\sigma\left(Y_{0}\right)=\emptyset\).  Then for \(\lambda \in \mathbb{C}\), we have the resolvent 
	$$
	R\left(\lambda, \mathcal{A}_{1}\right)=\left(\begin{array}{cccc}
		R\left(\lambda, Y_{0}\right) & \epsilon_{\lambda} & 0 & 0 \\
		0 & 0 & 0 & 0 \\
		0 & 0 & R\left(\lambda, \mathfrak{A}_{0}\right) & \varphi_{\lambda} \\
		0 & 0 & 0 & 0
	\end{array}\right),
	$$
	where
	\begin{equation}\label{3.5}
		\epsilon_{\lambda}(\tau)=e^{\lambda \tau}, \tau \in[-\theta, 0] \text { and } \varphi_{\lambda}(s)=e^{-\int_{0}^{s} \frac{\lambda+\nu_{*}(y)}{\gamma_{*}(y)} \ud y} , s \in [0,m].
	\end{equation}
	In addition,
	\begin{align*}
		\operatorname{ker}\left(\lambda-Y_{m}\right) &=\left\{f \cdot \epsilon_{\lambda}: f \in \mathfrak{X} \right\}, \\
		\operatorname{ker}\left(\lambda-\mathfrak{A}_{m}\right) &=<\varphi_{\lambda}>.
	\end{align*}
	Let $(z_{1} \hspace{0.3em}  z_{2} \hspace{0.3em} z_{3} \hspace{0.3em} z_{4})^{T} \in \mathcal{X}$ and $\lambda>0$, we have
	\begin{align*}
		\left\|R\left(\lambda, \mathcal{A}_{1}\right)(z_{1} \hspace{0.3em}  z_{2} \hspace{0.3em} z_{3} \hspace{0.3em} z_{4})^{T}\right\| =&\left\|R\left(\lambda, Y_{0}\right) z_{1}+\epsilon_{\lambda} z_{2}\right\|_{E}+\left\|R\left(\lambda, \mathfrak{A}_{0}\right) z_{3}+z_{4} \varphi_{\lambda}\right\|_{\mathfrak{X}} \\
		\leq &\left\|R\left(\lambda, Y_{0}\right) z_{1}\right\|_{E}+\left\|\epsilon_{\lambda} z_{2}\right\|_{E}+\left\|R\left(\lambda, \mathfrak{A}_{0}\right) z_{3}\right\|_{\mathfrak{X}}+\left\|z_{4} \varphi_{\lambda}\right\|_{\mathfrak{X}} \\
		\leq & \int_{-\theta}^{0} \frac{1}{\lambda}\|z_{1}(\tau)\|_{\mathfrak{X}} d \tau+\frac{1}{\lambda}\left\|z_{2}\right\|_{\mathfrak{X}}+\frac{1}{\lambda}\left\|z_{3}\right\|_{\mathfrak{X}}+\frac{1}{\lambda}|z_{4}| \\
		=& \frac{1}{\lambda}\left(\|z_{1}\|_{E}+\left\|z_{2}\right\|_{\mathfrak{X}}+\left\|z_{3}\right\|_{\mathfrak{X}}+|z_{4}|\right).
	\end{align*}
	Therefore, we obtain
	$$
	\left\|\lambda R\left(\lambda, \mathcal{A}_{1}\right)\right\| \leq 1,
	$$
	and $\mathcal{A}_{1}$ is a Hille-Yosida operator.
	Since the perturbing operator $\mathcal{A}_{2}$ is bounded, it follows from Lemma \ref{lemma2.4} that $\mathcal{A}$ is also a Hille-Yosida operator. In particular, the Hille-Yosida  operator $\mathcal{A}$  is the generator of a strongly continuous semigroup on the closure of its domain, by Lemma \ref{theorem2.2}.
\end{proof}

Hence, according to Proposition \ref{prop3.2} we observe that the operator \(\left(\mathcal{A}_{0}, D\left(\mathcal{A}_{0}\right)\right)\) also yields a strongly continuous semigroup on the space \(E \times\{0\} \times \mathfrak{X} \times\{0\}\). The operator \(\left(\mathscr{A}_{1}, D\left(\mathscr{A}_{1}\right)\right)\) generates a \(C_{0}\)-semigroup on \(\mathscr{X}\), as shown by the following theorem.

\begin{theorem}\label{theorem3.4}
	The operator \(\left(\mathscr{A}_{1}, D\left(\mathscr{A}_{1}\right)\right)\) is isomorphic to the part \(\left(\mathcal{A}_{0}, D\left(\mathcal{A}_{0}\right)\right)\) of the operator \((\mathcal{A}, D(\mathcal{A}))\) on the closure of its domain \(\overline{D(\mathcal{A})}\).
\end{theorem}

\begin{proof}
	From Definition \ref{defination2.3}, we observe that the part $\left(\mathcal{A}_{0}, D\left(\mathcal{A}_{0}\right)\right)$ of $(\mathcal{A}, D(\mathcal{A}))$ on the closure of its domain
	$$
	\mathcal{X}_{0}:=\overline{D(\mathcal{A})}=E \times\{0\} \times \mathfrak{X} \times\{0\}
	$$
	generates a strongly continuous semigroup. Or more precisely, \begin{align*}
		D\left(\mathcal{A}_{0}\right)&=\{x \in D(\mathcal{A}): \mathcal{A} x \in \overline{D(\mathcal{A})}\}\\
		&=\left\{ \hspace{0.3em} \left(\begin{array}{l}y\\ 0 \\s \\ 0\end{array}\right): y \in D\left(Y_{m}\right), s\in D\left(\mathfrak{A}_{m}\right), \hspace{0.3em} \mathcal{A}\left(\begin{array}{l}y \\ 0 \\ s \\ 0\end{array}\right) \in \mathcal{X}_{0} \hspace{0.05em} \right\}\\
		&=\left\{ \hspace{0.2em} \left(\begin{array}{l}y\\ 0 \\ s\\ 0\end{array}\right): y \in D\left(Y_{m}\right), s\in D\left(\mathfrak{A}_{m}\right), \hspace{0.3em} \begin{array}{l}
			Gy=s \\
			\mathcal{P}s=\Phi y
		\end{array} \hspace{0.05em} \right\}.
	\end{align*}
	
	Hence, the operator $(\mathscr{A}_{1}, D(\mathscr{A}_{1}))$ is isomorphic to $\left(\mathcal{A}_{0}, D\left(\mathcal{A}_{0}\right)\right)$ and generates a $C_{0}$-semigroup on the state space $\mathcal{X}$.
\end{proof}
Next we formulate the most important result of this section as folllows.
\begin{theorem}\label{HY}
	The operator \((\mathscr{A}, D(\mathscr{A}))\) of the abstract boundary delay problem \eqref{uu} generates a strongly continuous semigroup \((\mathscr{T}(t))_{t \geqslant 0}\) of boundary linear operators on \(\mathscr{X}\).
\end{theorem}
\begin{proof}
	Since isomorphisms have similar properties, we can obtain that the matrix operator \(\mathscr{A}_{1}\) is as well as  a Hille-Yosida operator. 
	In addition to this, both \((\mathscr{A}_{2}, D(\mathscr{A}_{2}))\) and \(\mathfrak{B}_{m}\) are  bounded
	perturbations of \(\mathscr{A}\) on \(\mathfrak{X}\), thus by \(\mathscr{A}=\mathscr{A}_{1}+\mathscr{A}_{2}\) and using the Desch-Schappacher perturbation theorem (Corollary 3.4 in \cite{KJ2000}),  we conclude that  \(\mathscr{A}\) generates a strongly continuous semigroup.
\end{proof}
The following well-posedness result for \eqref{uu} is implied by Theorem \ref{HY} (see Theorem 2.1 in Ref. \cite{SP2001}).

\begin{proposition}
	Assume that the initial value of the linear boundary delay problem \eqref{cauchy} is \(u^{0} \in E\), then it has a unique solution \(u(s, t)\) in the space \(C\left([-\theta,+\infty), \mathfrak{X}\right)\), given by \(u(s, t)=u^{0}(s, t)\) for \(t \in[-\theta, 0]\) and
	$$
	u(s, t)=\Pi_{2}\left(\mathscr{T}(t)\left(\begin{array}{c}
		u^{0}(s, 0) \\
		u^{0}(0)
	\end{array}\right)\right), \text { for } t>0,
	$$
	where \(\Pi_{2}\) is the projection operator of $\mathscr{T}(t)$ on the space $\mathfrak{X}$.
\end{proposition}

\section{Regularity properties of the $C_{0}$-semigroup}

In this section, we study regularity properties of the governing linear semigroup and use results from the spectral theory of $C_0$ semigroups to prove that $s\left(\mathscr{A}\right) \in \sigma\left(\mathscr{A}\right)=\sigma_{p}\left(\mathscr{A}\right) $. Then the stability of the positive stationary solution  of model \eqref{1.1} is determined by the position of the leading  eigenvalue. We will then demonstrate that it is possible to obtain an explicit characteristic equation corresponding to the linearised system to determine the position of the  leading eigenvalue. 

We first establish the main result of this section.

\begin{theorem}\label{theorem4.1}
	The spectrum of $\mathscr{A}$ can contain only isolated eigenvalues of finite multiplicity.
\end{theorem}

\begin{proof}
	Since the operator \(\mathscr{A}_{2}\) is clearly compact on \(\mathfrak{X}\), it suffices to verify the claim for the operator \(\mathscr{A}_{1}\). To this end, given $z \in \mathfrak{X}$, we find a unique solution $u \in D(\mathscr{A}_{1})$ of the equation
	$$
	\lambda u-\mathscr{A}_{1} u=z
	$$
	in the form
	\begin{equation}
		u(s)=e^{-\int_{0}^{s} \frac{\lambda+\nu_{*}(y)}{\gamma_{*}(y)} \ud y} \int_{0}^{s} e^{\int_{s}^{\alpha} \frac{\lambda+\nu_{*}(a)}{\gamma_{*}(a)} \ud a}  \frac{z(\alpha)}{\gamma_{*}(\alpha)} \ud \alpha. 
	\end{equation}
	Consequently, for $\lambda>0$ large enough, the resolvent operator $(\lambda I-\mathscr{A}_{1})^{-1}$ exists and is bounded, mapping $\mathfrak{X}=L^{1}(0, m)$ into $W^{1,1}(0, m)$. It then follows from Sobolev embedding theorems, that $W^{1,1}(0, m)$ is compactly embedded in  $\mathfrak{X}$ , that is, any bounded set $M$ on $W^{1,1}(0, m)$ is a compact set on  $\mathfrak{X}$. It also follows from the boundedness of $(\lambda I-\mathscr{A}_{1})^{-1}$, that $(\lambda I-\mathscr{A}_{1})^{-1} M $ is a bounded set in $W^{1,1}(0, m)$. Using the definition of a compact operator, it's not hard to see that $(\lambda I-\mathscr{A}_{1})$ is a compact operator on $W^{1,1}(0, m)$. The conclusion of the theorem is then obtained by using Riesz-Schauder theory, e.g. Lemma \ref{lemma2.8}.
\end{proof}

\begin{proposition}
	The linear stability of the stationary solution of model \eqref{1.1} is determined by spectrum of the generator, i.e.,
	$$
	\sigma(\mathscr{T}(t))=\{0\} \cup e^{\sigma(\mathscr{A})}, \quad t>0.
	$$
	Furthermore, the spectral bound $$ s(\mathscr{A})=\sup \{\mathrm{Re} \lambda \mid \lambda \in \sigma (\mathscr{A})$$ 
	coincides with the growth rate (see e.g. \cite{KJ2000,AP1983})  $$\omega_{0}=\displaystyle\lim _{t \rightarrow \infty} t^{-1} \ln \|\mathscr{T}(t)\|.$$
\end{proposition}

If the eigenvalue with the largest real part was real, our analysis would be greatly simplified. In some cases, the following finding allows us to reach this conclusion. To this end, we would mention various existing lemmas and theorems in order to establish the  second major result in this section.

\begin{lemma}\label{le4.3}
	For \(\lambda \in \rho\left(Y_{0}\right) \cap \rho\left(\mathfrak{A}_{0}\right)\), we define  the abstract Dirichlet
	operators (see e.g. \cite{RN1990}) separately
	\begin{equation}
		\begin{aligned}
			&K_{\lambda}: \mathfrak{X}\rightarrow E  \text{ by } K_{\lambda}:=1 \circ \epsilon_{\lambda},  \\
			&L_{\lambda}: E \rightarrow \mathfrak{X} \text{ by } L_{\lambda}:=\left(1 \circ \varphi_{\lambda}\right) \Phi, 
		\end{aligned}
	\end{equation}
	where \(\epsilon_{\lambda}\) and \(\varphi_{\lambda}\) are given in \eqref{3.5}. 
	Then  \(K_{\lambda} \in \mathscr{L}(\mathfrak{X}, E)\) and \(L_{\lambda} \in \mathscr{L}(E, \mathfrak{X}) .\) Apart from that,
	\begin{equation}
		\begin{aligned}
			& G\left(K_{\lambda}(f)\right)=f, \text { for all } 
			f  \in D\left(\mathfrak{A}_{m}\right),\\
			& \mathcal{P}\left(L_{\lambda}\left(y\right)\right)=\Phi(y),  \text { for all } y \in D\left(Y_{m}\right).
			\\ 
		\end{aligned}
	\end{equation}  
\end{lemma}
Next we will study the position of eigenvalues related to the compactness of operators, in particular we have:
\begin{lemma}\label{th4.5} 
	Let \(\lambda \in \rho\left(Y_{0}\right) \cap \rho\left(\mathfrak{A}_{0}\right)\), and consider  the following properties\\
	(i) \(\lambda \in \rho\left(\mathscr{A}_{1}\right)\);\\
	(ii) \(1 \in \rho\left(K_{\lambda} L_{\lambda}\right)\) for the operator \(K_{\lambda} L_{\lambda} \in \mathscr{L}(E)\);\\
	(iii) \(1 \in \rho\left(L_{\lambda} K_{\lambda}\right)\) for the operator \(L_{\lambda} K_{\lambda} \in \mathscr{L}(\mathfrak{X})\).\\
	Then one has the implications \((i) \Leftarrow(ii) \Leftrightarrow(iii)\).  In particular, if \(K_{\lambda}\) and \(L_{\lambda}\) are compact operators,  the assertions \((i),(ii)\) and \((iii)\) are equivalent.
\end{lemma}
This lemma is taken from \cite{RN1990}, specifically see Theorem 2.7 in \cite{RN1990}.
Here the operator \(L_{\lambda}\) is compact, which has one-dimensional range. Therefore \(K_{\lambda} L_{\lambda}\) and \(L_{\lambda} K_{\lambda}\) are compact too.  From Lemma \ref{th4.5}  we have the following result.
\begin{theorem}\label{theo4.6}
	For the operator $(\mathscr{A}_{1}, D(\mathscr{A}_{1}))$, there holds that\\
	(i) \(\lambda \in \sigma\left(\mathscr{A}_{1}\right) \Leftrightarrow 1 \in \sigma\left(L_{\lambda} K_{\lambda}\right) \Leftrightarrow 1 \in \sigma_{p}\left(L_{\lambda} K_{\lambda}\right) \Leftrightarrow \lambda \in \sigma_{p}\left(\mathscr{A}_{1}\right)\);\\
	(ii) Moreover, if \(\lambda \in \rho\left(\mathscr{A}_{1}\right)\) equivalently \(1 \in \rho\left(L_{\lambda} K_{\lambda}\right)\), then the resolvent of $\mathscr{A}_{1}$ is given by
	\begin{equation}\label{th4.6 4.3}
		R\left(\lambda, \mathscr{A}_{1}\right)=\left(\begin{array}{cc}
			\left(1-K_{\lambda} L_{\lambda}\right)^{-1} R\left(\lambda, Y_{0}\right) & \left(1-K_{\lambda} L_{\lambda}\right)^{-1} K_{\lambda} R\left(\lambda, \mathfrak{A}_{0}\right) \\
			\left(1-L_{\lambda} K_{\lambda}\right)^{-1} L_{\lambda} R\left(\lambda, Y_{0}\right) & \left(1-L_{\lambda} K_{\lambda}\right)^{-1} R\left(\lambda, \mathfrak{A}_{0}\right)
		\end{array}\right).
	\end{equation}
\end{theorem}
\begin{proof}
	We just need to verify \eqref{th4.6 4.3}.  For \(\lambda \in \rho\left(Y_{0}\right) \cap \rho\left(\mathfrak{A}_{0}\right)\), we have
	\begin{equation}\label{le4.4 4.2}
		\left(\lambda-\mathscr{A}_{1}\right)=\left(\begin{array}{cc}
			\lambda-Y_{0} & 0 \\
			0 & \lambda-\mathfrak{A}_{0}
		\end{array}\right) \mathcal{B}_{\lambda},
	\end{equation}
	where \(\mathcal{B}_{\lambda}:=\left(\begin{array}{cc}\text { Id } & -K_{\lambda} \\ -L_{\lambda} & I d\end{array}\right)\) is a bounded linear matrix operator on \(D\left(Y_{m}\right) \times\)
	\(D\left(\mathfrak{A}_{m}\right)\) and the matrix \(\left(\begin{array}{cc}\lambda-Y_{0} & 0 \\ 0 & \lambda-\mathfrak{A}_{0}\end{array}\right)\) has domain \(D\left(Y_{0}\right) \times D\left(\mathfrak{A}_{0}\right)\). 
	The inverse of \(\left(\lambda-\mathscr{A}_{1}\right)\) is
	$$
	R\left(\lambda, \mathscr{A}_{1}\right)=\mathcal{B}_{\lambda}^{-1}\left(\begin{array}{cc}
		R\left(\lambda, Y_{0}\right) & 0 \\
		0 & R\left(\lambda, \mathfrak{A}_{0}\right)
	\end{array}\right).
	$$
	By the definition of \(\mathcal{B}_{\lambda}\), we get
	$$
	\mathcal{B}_{\lambda}^{-1}=\left(\begin{array}{cc}
		\left(1-K_{\lambda} L_{\lambda}\right)^{-1} & \left(1-K_{\lambda} L_{\lambda}\right)^{-1} K_{\lambda} \\
		\left(1-L_{\lambda} K_{\lambda}\right)^{-1} L_{\lambda} & \left(1-L_{\lambda} K_{\lambda}\right)^{-1}
	\end{array}\right).
	$$
	Therefore expression \eqref{th4.6 4.3} follows.
\end{proof}

We conclude this section by establishing a criterion to guarantee the positivity of the governing linear semigroup.  
\begin{theorem}\label{4.8}
	Suppose that
	\begin{equation}\label{4.2}
		\begin{aligned}
			&\int_{-\theta}^{0} \beta\left(\cdot, \tau, Q_{*}(\cdot)\right) \ud \tau+ w(\cdot)\left(\int_{0}^{\cdot} \int_{-\theta}^{0} \beta_{Q}\left(y, \tau, Q_{*}(y)\right) p_{*}(y) \ud \tau \ud y\right. \\
			&\left.+\alpha \int_{\cdot}^{m} \int_{-\theta}^{0} \beta_{Q}\left(y, \tau, Q_{*}(y)\right) p_{*}(y) \ud \tau \ud y\right) \geq 0,
		\end{aligned}
	\end{equation}
	then the semigroup \((\mathscr{T}(t))_{t\geqslant 0} \), generated by the operator \((\mathscr{A}, D(\mathscr{A}))\) is positive.
\end{theorem} 

\begin{proof}
	Here condition \eqref{4.2} is a direct generalisation of the positivity condition corresponding to the age-structured model established in \cite{JP1983}. 
	If $\beta_{Q}\equiv 0$,  condition \eqref{4.2} is trivially satisfied.
	Condition \eqref{4.2} guarantees that operator \(\mathscr{A}_{2}\) is positive, then we only need to show that the semigroup \((\mathscr{T}_{1}(t))_{t \geq 0}\) generated by the operator \(\mathscr{A}_{1}\) is positive.
	Firstly, we consider the operator \(K_{\lambda} L_{\lambda}\). By the definitions of \(K_{\lambda}\) and \(L_{\lambda}\) in Lemma \ref{le4.3}, it is clear that
	$$
	\lim _{\operatorname{Re\lambda} \rightarrow+\infty}\left\|K_{\lambda} L_{\lambda}\right\|=0.
	$$
	For \(\mathrm{Re} \lambda \) is sufficiently large, we have  \(\left\|K_{\lambda} L_{\lambda}\right\|<1\). 
	The operator \(\left(1-K_{\lambda} L_{\lambda}\right)\) is invertible, and the Neumann series determines its inverse \(\left(1-K_{\lambda} L_{\lambda}\right)^{-1}\). Distinctly, the condition \eqref{4.2}  implies that \(K_{\lambda} L_{\lambda}\) is a positive operator, and \(\left(1-K_{\lambda} L_{\lambda}\right)^{-1}\) is positive as well if \(\mathrm{Re}  \lambda\) is large enough. 
	Hence from the representation \eqref{th4.6 4.3}, we see that \(R\left(\lambda, \mathscr{A}_{1}\right)\) is non-negative for such $\lambda$. Therefore, in combination with Lemma \ref{le2.7} above, the operator \(\left(\mathscr{A}_{1}, D\left(\mathscr{A}_{1}\right)\right)\) generates a positive semigroup on the Banach lattice \(E \times \mathfrak{X}\), which concludes the proof.
\end{proof}

The following result can be established using results from the theory of positive semigroups (see e.g. \cite{KJ2000,AP1983} and also \cite{F2007,FH2008} for
similar results).

\begin{proposition} \label{coro4.7}
	Suppose that condition \eqref{4.2} is satisfied. Then  $s(\mathscr{A}) \in \sigma(\mathscr{A})$. Specifically, $s(\mathscr{A})$ is a dominant  eigenvalue, namely
	$$
	s(\mathscr{A})=\sup \{\mathrm{Re} \lambda \mid \lambda \in \sigma_{p}(\mathscr{A})\}.
	$$
\end{proposition}

\section{The characteristic equation}

The linear stability of stationary solutions of
model \eqref{1.1} is determined by the eigenvalues of the semigroup generator $\mathscr{A}$ according to the results we derived in the previous section. In this section, we derive an explicit  characteristic equation to study the position of the eigenvalues of the generator $\mathscr{A}$.

The eigenvalue equation
\begin{equation}\label{5.1}
	(\lambda I-\mathscr{A}) u=0
\end{equation}
for $\lambda \in \mathbb{C}$ and non-trivial $u$ is equivalent to the system
\begin{equation}{\label{5.2}}
	\begin{aligned}
	0= &\gamma_{*}(s) u^{\prime}(s)+(\lambda+\nu_{*}(s)) u(s)+ \varepsilon_{*}(s) \bar{U},
	\\[5pt]    
	u(0)= & \int_{0}^{m} \int_{-\theta}^{0} e^{\lambda \tau}\left(\beta(s, \tau, Q_{*}(s)) u(s)+\beta_{Q}(s, \tau, Q_{*}(s)) p_{*}(s) H(s)\right) \ud \tau \ud s,
	\end{aligned}
\end{equation}
where $\bar{U}= \int_{0}^{m} u(s) \ud s$ and
\begin{equation}\label{5.5}
	\begin{aligned}
		H(s) &=\alpha \int_{0}^{s} w(r) u(r) \ud r+\int_{s}^{m} w(r) u(r) \ud r \\
		&=(\alpha-1) \int_{0}^{s} w(r) u(r) \ud r+\int_{0}^{m} w(r) u(r) \ud r.
	\end{aligned}
\end{equation}

\noindent We assume that $\alpha \in[0,1)$ holds for the rest of this section. From \eqref{5.5},  we have
\begin{equation}\label{5.6}
	H^{\prime}(s)=(\alpha-1) w(s) u(s) \text { and } H^{\prime \prime}(s)=(\alpha-1)\left(w^{\prime}(s) u(s)+w(s) u^{\prime}(s)\right).
\end{equation}
Using relations \eqref{5.6}, we can write  system \eqref{5.2} in the form of $H$ as well as its derivatives 
\begin{equation}\label{5.7}
	H^{\prime \prime}(s)+\left(\frac{\lambda+\nu_{*}(s)}{\gamma_{*}(s)}-\frac{w^{\prime}(s)}{w(s)}\right) H^{\prime}(s)+(\alpha-1)\bar{U}\frac{w(s)\varepsilon_{*}(s)}{\gamma_{*}(s)}=0.
\end{equation}
Eq. \eqref{5.7} is accompanied by boundary conditions of the form
\begin{equation}\label{5.8}
	\alpha H(0)=H(m),
\end{equation}
\begin{equation}\label{5.9}
	\begin{aligned}
		H^{\prime}(0)=&w(0) \int_{0}^{m} \int_{-\theta}^{0} e^{\lambda \tau} \frac{\beta\left(s, \tau, Q_{*}(s)\right)}{w(s)} H^{\prime}(s) \ud \tau \ud s \\
		&+(\alpha-1) w(0) \int_{0}^{m} \int_{-\theta}^{0} e^{\lambda \tau} \beta_{Q}\left(s, \tau, Q_{*}(s)\right) p_{*}(s) H(s) \ud \tau \ud s .
	\end{aligned}
\end{equation}
Thus the general solution of \eqref{5.7} can be expressed as
\begin{equation}\label{5.10}
	H(s)=H(0)+H^{\prime}(0) \int_{0}^{s} \frac{w(y)}{w(0)} \pi_{*}(\lambda, y) \d y+(1-\alpha)\int_{0}^{s} w(y) \pi_{*}(\lambda, y)\int_{0}^{y}\frac{\varepsilon_{*}(r) \bar{U}}{\pi_{*}(\lambda, y)\gamma_{*}(r)} \ud r \ud y,
\end{equation}
where
$$
\pi_{*}(\lambda, y)=e^{-\int_{0}^{y} \frac{\lambda+\gamma_{s}(a,P_{*})+\mu(a,P_{*})}{\gamma(a,P_{*})} \d a}.
$$
Meanwhile, substituting the solution \eqref{5.10} into \eqref{5.9},  we get
\begin{equation}\label{5.11}
	\begin{aligned}
		0&=
		H(0)(\alpha-1) w(0) \int_{0}^{m} \int_{-\theta}^{0} e^{\lambda \tau} \beta_{Q}\left(s, \tau, Q(s)_{*}(s)\right) p_{*}(s)  \ud \tau \ud s\\
		&+
		H^{\prime}(0)\left(1-\int_{0}^{m}\int_{-\theta}^{0}e^{\lambda\tau}\beta(s,\tau,Q_{*}(s))\pi_{*}(\lambda, s)\ud \tau \ud s\right)\\
		&+
		H^{\prime}(0)(1-\alpha) \int_{0}^{m} \int_{-\theta}^{0} e^{\lambda \tau} \beta_{Q}\left(s, \tau, Q_{*}(s)\right) p_{*}(s)\int_{0}^{s} w(y)\pi_{*}(\lambda, y)\ud y  \ud \tau \ud s\\
		&-
		\bar{U}\int_{0}^{m} \int_{-\theta}^{0}e^{\lambda\tau}\beta(s,\tau,Q_{*}(s))\pi_{*}(\lambda, s) \int_{0}^{s}   \frac{\varepsilon_{*}(y)w(0)(1-\alpha) }{\pi_{*}(\lambda, y)\gamma_{*}(y)} \ud y \ud \tau \ud s\\
		&+
		\bar{U}
		\int_{0}^{m}\int_{-\theta}^{0}\left(e^{\lambda\tau}\beta_{Q}(s,\tau,Q_{*}(s))p_{*}(s)\int_{0}^{s} w(y)\pi_{*}(\lambda, y) \int_{0}^{y}\frac{\varepsilon_{*}(r)w(0)(1-\alpha)^{2} }{\pi_{*}(\lambda, r)\gamma_{*}(r)} \ud r \ud y \right) \ud \tau \ud s.
	\end{aligned}
\end{equation}
Using the boundary condition \eqref{5.8} and the solution \eqref{5.10}, we obtain
\begin{equation}\label{5.12}
	(1-\alpha) H(0)+H^{\prime}(0) \int_{0}^{m} \frac{w(s)}{w(0)} \pi_{*}(\lambda, s) \ud s + \bar{U} \int_{0}^{m} w(s) \pi_{*}(\lambda, s)\int_{0}^{s}\frac{\varepsilon_{*}(y)(1-\alpha) }{\pi_{*}(\lambda, y)\gamma_{*}(y)} \ud y \ud s =0.
\end{equation}
The general solution of Eq.$\eqref{5.2}_{a}$ takes the form
\begin{equation}\label{5.13}
	u(s)=u(0) \pi_{*}(\lambda, s)-\pi_{*}(\lambda, s) \int_{0}^{s}\frac{\varepsilon_{*}(y)\bar{U}}{\pi_{*}(\lambda, y)\gamma_{*}(y) } \ud y .
\end{equation}
Integrating \eqref{5.13} from 0 to \(m\), we obtain
\begin{equation}\label{5.14}
	\begin{aligned}
		\bar{U}=&-\bar{U}\int_{0}^{m} \pi_{*}(\lambda, s)
		\int_{0}^{s}\frac{\varepsilon_{*}(y)}{ \pi_{*}(\lambda, y)\gamma_{*}(y)} \ud y \ud s +U(0) \int_{0}^{m} \pi_{*}(\lambda, s) \ud s.  
	\end{aligned}
\end{equation}
Eqs.\eqref{5.11}, \eqref{5.14} and the boundary condition \eqref{5.6} imply that
\begin{equation}\label{5.15}
	\begin{aligned}
		0&=
		H(0) \int_{0}^{m}\pi_{*}(\lambda, s) \ud s  \int_{0}^{m} \int_{-\theta}^{0}  e^{\lambda \tau} \beta_{Q}\left(s, \tau, Q_{*}(s)\right) p_{*}(s)  \ud \tau \ud s\\
		&+
		H^{\prime}(0)\int_{0}^{m}\pi_{*}(\lambda, s) \ud s \int_{0}^{m}\int_{-\theta}^{0}e^{\lambda\tau}\beta(s,\tau,Q_{*}(s))\frac{\pi_{*}(\lambda, s)}{w(0)(\alpha-1)}\ud \tau \ud s\\
		&+
		H^{\prime}(0)\int_{0}^{m}\pi_{*}(\lambda, s) \ud s\int_{0}^{m} \int_{-\theta}^{0} e^{\lambda \tau} \beta_{Q}\left(s, \tau, Q_{*}(s)\right) p_{*}(s)\int_{0}^{s} \frac{w(y)}{w(0)}\pi_{*}(\lambda, y)\ud y  \ud \tau \ud s\\
		&+
		\bar{U}\int_{0}^{m}\pi_{*}(\lambda, s) \ud s 
		\int_{0}^{m} \int_{-\theta}^{0}e^{\lambda\tau}\beta_{Q}(s,\tau,Q_{*}(s))p_{*}(s) \int_{0}^{s}  \int_{0}^{y} \frac{\varepsilon_{*}(r)(1-\alpha) }{\pi_{*}(\lambda, r)\gamma_{*}(r)}w(y) \pi_{*}(\lambda, y)\ud r \ud y \ud \tau \ud s\\
		&-
		\bar{U}\int_{0}^{m}\pi_{*}(\lambda, s) \ud s 
		\int_{0}^{m}\int_{-\theta}^{0}e^{\lambda\tau}\beta(s,\tau,Q_{*}(s))\pi_{*}(\lambda, s)\int_{0}^{s}\frac{\varepsilon_{*}(y) }{\pi_{*}(\lambda, y)\gamma_{*}(y)} \ud y \ud \tau \ud s\\
		&-
		\bar{U}\left(1+\int_{0}^{m}\pi_{*}(\lambda, s) \int_{0}^{s}\frac{\varepsilon_{*}(y) }{\pi_{*}(\lambda, y)\gamma_{*}(y)} \ud y \ud s \right).
	\end{aligned}
\end{equation}
Hence, the linear system composed of \eqref{5.11}, \eqref{5.12} and \eqref{5.15} has a non-zero solution $(H(0),H^{\prime}(0),\bar{U})$ if and only if $\lambda$ satisfies the equation 
\begin{equation}\label{5.16}
	\left(\begin{array}{lll}A_{11}(\lambda) & A_{12}(\lambda) & A_{13}(\lambda) \\ A_{21}(\lambda) & A_{22}(\lambda) & A_{23}(\lambda)\\  A_{31}(\lambda) & A_{32}(\lambda) & A_{33}(\lambda) \end{array}\right)\left(\begin{array}{c} H(0)\\ H^{\prime}(0) \\ \bar{U}(0)\end{array}\right)=0,
\end{equation}
where we define
\begin{align*}
	A_{11}(\lambda) =& (1-\alpha) w(0) \int_{0}^{m} \int_{-\theta}^{0} e^{\lambda \tau} \beta_{Q}\left(s, \tau, Q_{*}(s)\right) p_{*}(s)  \ud \tau \ud s,\\
	A_{12}(\lambda) =& 1-\int_{0}^{m}\int_{-\theta}^{0}e^{\lambda\tau}\beta(s,\tau,Q_{*}(s))\pi_{*}(\lambda, s)\ud \tau \ud s\\
	&+
	(1-\alpha) \int_{0}^{m} \int_{-\theta}^{0} e^{\lambda \tau} \beta_{Q}\left(s, \tau, Q_{*}(s)\right) p_{*}(s)\int_{0}^{s} w(y)\pi_{*}(\lambda, y)\ud y  \ud \tau \ud s,\\
	A_{13}(\lambda) =& \int_{0}^{m} \int_{-\theta}^{0}e^{\lambda\tau}\beta(s,\tau,Q_{*}(s))\pi_{*}(\lambda, s) \int_{0}^{s}   \frac{\varepsilon_{*}(y)w(0)(\alpha-1) }{\pi_{*}(\lambda, y)\gamma_{*}(y)} \ud y \ud \tau \ud s\\
	&+
	\int_{0}^{m}\int_{-\theta}^{0}\left(e^{\lambda\tau}\beta_{Q}(s,\tau,Q_{*}(s))p_{*}(s)\int_{0}^{s} w(y)\pi_{*}(\lambda, y) \int_{0}^{y}\frac{\varepsilon_{*}(r)w(0)(1-\alpha)^{2} }{\pi_{*}(\lambda, r)\gamma_{*}(r)} \d r \d y \right) \ud \tau \ud s,\\ 
	A_{21}(\lambda) =&1-\alpha,\\
	A_{22}(\lambda) =& \int_{0}^{m} \frac{w(s)}{w(0)} \pi_{*}(\lambda, s) \ud s, \\
	A_{23}(\lambda) =& \int_{0}^{m} w(s) \pi_{*}(\lambda, s)\int_{0}^{s}\frac{\varepsilon_{*}(y)(1-\alpha) }{\pi_{*}(\lambda, y)\gamma_{*}(y)} \ud y \ud s,\\
	A_{31}(\lambda) =&\int_{0}^{m}\pi_{*}(\lambda, s) \ud s  \int_{0}^{m} \int_{-\theta}^{0}  e^{\lambda \tau} \beta_{Q}\left(s, \tau, Q_{*}(s)\right) p_{*}(s)  \ud \tau \ud s,\\
	A_{32}(\lambda) =&\int_{0}^{m}\pi_{*}(\lambda, s) \ud s \int_{0}^{m}\int_{-\theta}^{0}e^{\lambda\tau}\beta(s,\tau,Q_{*}(s))\frac{\pi_{*}(\lambda, s)}{w(0)(\alpha-1)}\ud \tau \ud s\\
	&+
	\int_{0}^{m}\pi_{*}(\lambda, s) \ud s\int_{0}^{m} \int_{-\theta}^{0} e^{\lambda \tau} \beta_{Q}\left(s, \tau, Q_{*}(s)\right) p_{*}(s) \int_{0}^{s} \frac{w(y)}{w(0)}\pi_{*}(\lambda, y)\ud y   \ud \tau \ud s,\\
	A_{33}(\lambda) =& \int_{0}^{m}\pi_{*}(\lambda, s) \ud s 
	\int_{0}^{m} \int_{-\theta}^{0}e^{\lambda\tau}\beta_{Q}(s,\tau,Q_{*}(s))p_{*}(s) \int_{0}^{s}  \int_{0}^{y} \frac{\varepsilon_{*}(r)(1-\alpha) }{\pi_{*}(\lambda, r)\gamma_{*}(r)}w(y) \pi_{*}(\lambda, y)\ud r \ud y \ud \tau \ud s\\
	&-
	\int_{0}^{m}\pi_{*}(\lambda, s) \ud s 
	\int_{0}^{m}\int_{-\theta}^{0}e^{\lambda\tau}\beta(s,\tau,Q_{*}(s))\pi_{*}(\lambda, s)\int_{0}^{s}\frac{\varepsilon_{*}(y) }{\pi_{*}(\lambda, y)\gamma_{*}(y)} \ud y \ud \tau \ud s\\
	&-
	\left(1+\int_{0}^{m}\pi_{*}(\lambda, s) \int_{0}^{s}\frac{\varepsilon_{*}(y) }{\pi_{*}(\lambda, y)\gamma_{*}(y)} \ud y \ud s\right).
\end{align*}

\begin{proposition}
	$\lambda \in \mathbb{C} $ is a eigenvalue of the operator $\mathscr{A}$ if and only if $\lambda$ is a solution of the following characteristic equation
	\begin{equation}\label{char.}
		K(\lambda)=\left|\begin{array}{cccc}
			A_{11}(\lambda) & A_{12}(\lambda) & A_{13}(\lambda) \\
			A_{21}(\lambda) & A_{22}(\lambda) & A_{23}(\lambda)  \\
			A_{31}(\lambda) & A_{32}(\lambda) & A_{33}(\lambda) 
		\end{array}\right|=0.
	\end{equation}
\end{proposition}
In summary, $K(\lambda)$ determines the characteristic equation corresponding to the linearised system \eqref{2-2-linearized} and its zeros are the eigenvalues of the operator $\mathscr{A}$, which completely determine the spectrum of $\mathscr{A}$, and therefore the linear stability of the steady state.

\section{Linear stability analysis}

In the previous section we deduced an explicit characteristic equation corresponding to the linearisation of hierarchical size-structured model \eqref{1.1}. We now use this characteristic equation to derive stability criteria. By virtue of Corollary  
\ref{coro4.7} we can investigate the asymptotic stability and instability of stationary solutions of model \eqref{1.1} using the characteristic equation. Further, we will show how the basic reproduction function $\mathscr{R}(P,Q)$  introduced in Eq. \eqref{2.3} can be used to establish stability and instability conditions.

The first result addresses the stability of the trivial stationary solution \(p_{0}\equiv0\).
\begin{theorem}\label{th5.2}
	The trivial stationary solution \(p_{0}\equiv0\) is linearly asymptotically stable if
	$\mathscr{R}(0,0)<1$, and unstable if $\mathscr{R}(0,0)>1$ holds.
\end{theorem}
\begin{proof}
	For \(p_{0}\equiv0\) we have
	\begin{align*}
		\hat{A}_{12}(\lambda) =&\quad 1-\int_{0}^{m}\int_{-\theta}^{0}e^{\lambda\tau}\beta(s,\tau,0)\pi_{0}(\lambda, s)\ud \tau \ud s ,\\
		\hat{A}_{21}(\lambda)=& \quad 1-\alpha,\\
		\hat{A}_{22}(\lambda) =& \quad  \int_{0}^{m} \frac{w(s)}{w(0)} \pi_{0}(\lambda, s) \ud s, \\
		\hat{A}_{32}(\lambda)=& \quad \int_{0}^{m}\pi_{0}(\lambda, s) \ud s \int_{0}^{m}\int_{-\theta}^{0}e^{\lambda\tau}\beta(s,\tau,0)\frac{\pi_{0}(\lambda, s)}{w(0)(\alpha-1)}\ud \tau \ud s, \\
		\hat{A}_{31}(\lambda)=& \quad \hat{A}_{23}(\lambda)=  \hat{A}_{13}(\lambda)= \hat{A}_{11}(\lambda)=0,\\
		\hat{A}_{33}(\lambda)=& \quad -1,
	\end{align*} where  $$\pi_{0}(\lambda, s)=e^{-\int_{0}^{s} \frac{\lambda+\gamma_{s}(a,0)+\mu(a,0)}{\gamma(a,0)} \ud a}.$$
	Hence the characteristic equation \eqref{char.} reduces to
	\begin{equation}\label{K0}
		\begin{array}{c}
			\hat{K}(\lambda)=\left|\begin{array}{cccc}
				0 & \hspace{0.5em} \hat{A}_{12}(\lambda) &  \hspace{0.5em} 0\\
				1-\alpha & \hspace{0.5em} \hat{A}_{22}(\lambda) & \hspace{0.5em} 0\\
				0 & \hspace{0.5em} \hat{A}_{32}(\lambda) & \hspace{0.5em} -1 
			\end{array}\right| 
			= (1-\alpha)\hat{A}_{12}(\lambda), \text{ for } 0\leq \alpha <1.
		\end{array}
	\end{equation}
	It is readily observed that
	\begin{equation}
		\hat{K}(0)= (1-\alpha)\hat{A}_{12}(0)=(1-\alpha)\left(1-\mathscr{R}(0,0)\right).
	\end{equation}
	Clearly condition \eqref{4.2}  is satisfied and therefore we can restrict the characteristic equation $\hat{K}(\lambda)$ to $\lambda \in \mathbb{R}$. 
	Furthermore, from \eqref{K0}, we have
	\begin{equation}
		\lim _{\lambda \rightarrow+\infty} \hat{K}(\lambda)= 1-\alpha ,\hspace{1mm} \hat{K}^{\prime}(\lambda)=(1-\alpha)\hat{A}^{\prime}_{12}(\lambda)>0.
	\end{equation}
	Therefore, if $\mathscr{R}(0,0)<1$ holds, we can obtain $\hat{K}(0)>0$ and $\hat{K}(\lambda)$ is monotonically increasing, which implies that the characteristic equation \eqref{char.} cannot have non-negative roots. However, if $\mathscr{R}(0,0)>1$ holds, there is a positive root since $\hat{K}(0)<0$. The claim of the theorem follows.
\end{proof}

Next we will address the instability of positive stationary solutions.

\begin{theorem}\label{th5.3}
	Let \(p_{*}(s)\)  be any positive stationary solution of \eqref{1.1} and suppose that all the conditions of Theorem \ref{4.8} are fulfilled. Then the positive stationary solution  \(p_{*}(s)\) is linearly unstable if \(K(0)<0\).
\end{theorem}
\begin{proof}
	It suffices to show that there exists a positive solution \(\lambda\) of the characteristic equation \eqref{char.}. We can readily deduce that
	\begin{equation}
		\begin{array}{c}
			\lim _{\lambda \rightarrow+\infty} K(\lambda)=\begin{vmatrix}
				0 & 1 &  \hspace{0.5em} 0\\
				1- \alpha \hspace{0.5em} &  0 &  \hspace{0.5em} 0\\
				0 &  0 & \hspace{0.5em} -1
			\end{vmatrix}
			=1- \alpha, \text{ for  } 0\leq \alpha<1.
		\end{array}
	\end{equation}
	Here the limit is taken in $\mathbb{R}$, then we can formulate the above simple instability criterion, which follows immediately from the Intermediate Value Theorem since $K(0)<0$.
\end{proof}

Since a strict linear stability proof requires showing that all zeros of the characteristic equation are be located in  the left half-plane of $\mathbb{C}$, the stability results for positive stationary solutions of model are much more difficult to obtain than instability results, especially considering that our growth and mortality rates are both depend on the total population size, and the birth rate that involves fertility delay and an infinite dimensional interaction variable (environment). 
We will now demonstrate for some special cases of the model ingredients, that we can overcome these difficulties.
Consider the situation when  mortality and growth rates are independent of the population size $P$, i.e. $\gamma_{P}\equiv0\equiv\mu_{P}$. Hence $\varepsilon_{*}=p_{*}(s) \left(\mu_{P}(s,P_{*})+\gamma_{sP}(s,P_{*})\right)+p_{*}^{\prime}(s) \gamma_{P}(s,P_{*})=0$.
In this case, we can derive explicit conditions for the linear stability and instability of the positive stationary solution in a relatively straightforward fashion. We have the following result.

\begin{theorem}\label{th5.4}
	Suppose that $\varepsilon_{*} \equiv 0$ and the positivity condition \eqref{4.2} holds true.
	
	(i) If $ \beta_{Q}\left(s,\tau, Q_{*}\right) < 0$, then the positive stationary solution $p_{*}$ is linearly asymptotically stable.
	
	(ii) If $\beta_{Q}\left(s,\tau, Q_{*}\right) \geq 0$, then $p_{*}$ is linearly unstable.
\end{theorem}
\begin{proof}
	For the special case of model ingredients we are dealing now, we have for the terms in the characteristic equation \eqref{char.}
	\begin{align*}
		\tilde{A}_{11}(\lambda) =& (1-\alpha) w(0) \int_{0}^{m} \int_{-\theta}^{0} e^{\lambda \tau} \beta_{Q}\left(s, \tau, Q_{*}(s)\right) p_{*}(s)  \ud \tau \ud s,\\
		\tilde{A}_{12}(\lambda) =& 1-\int_{0}^{m}\int_{-\theta}^{0}e^{\lambda\tau}\beta(s,\tau,Q_{*}(s))\pi(\lambda, s)\ud \tau \ud s\\
		&+
		(1-\alpha) \int_{0}^{m} \int_{-\theta}^{0} e^{\lambda \tau} \beta_{Q}\left(s, \tau, Q_{*}(s)\right) p_{*}(s)\int_{0}^{s} w(y)\pi(\lambda, y)\ud y  \ud \tau \ud s,\\ 
		\tilde{A}_{13}(\lambda) =&0,\\ 
		\tilde{A}_{21}(\lambda) =&1-\alpha,\\
		\tilde{A}_{22}(\lambda) =& \int_{0}^{m} \frac{w(s)}{w(0)} \pi(\lambda, s) \ud s, \\
		\tilde{A}_{23}(\lambda) =& 0,\\
		\tilde{A}_{31}(\lambda) =&\int_{0}^{m}\pi(\lambda, s) \ud s  \int_{0}^{m} \int_{-\theta}^{0}  e^{\lambda \tau} \beta_{Q}\left(s, \tau, Q_{*}(s)\right) p_{*}(s)  \ud \tau \ud s,\\
		\tilde{A}_{32}(\lambda) =&\int_{0}^{m}\pi(\lambda, s) \ud s \int_{0}^{m}\int_{-\theta}^{0}e^{\lambda\tau}\beta(s,\tau,Q_{*}(s))\frac{\pi(\lambda, s)}{w(0)(\alpha-1)}\ud \tau \ud s\\
		&+
		\int_{0}^{m}\pi(\lambda, s) \ud s\int_{0}^{m} \int_{-\theta}^{0} e^{\lambda \tau} \beta_{Q}\left(s, \tau, Q_{*}(s)\right) p_{*}(s) \int_{0}^{s} \frac{w(y)}{w(0)}\pi(\lambda, y)\ud y   \ud \tau \ud s,\\
		\tilde{A}_{33}(\lambda) =&-1,
	\end{align*}
	where we have set  $$\pi(\lambda, s)=e^{-\int_{0}^{s} \frac{\lambda+\gamma_{s}(a)+\mu(a)}{\gamma(a)} \ud a}.$$
	It follows that
	$$
	\begin{aligned}
		\tilde{K}(\lambda)&= \int_{0}^{m}\int_{-\theta}^{0}e^{\lambda\tau}\beta(s,\tau,Q_{*}(s))\pi(\lambda, s)\ud \tau \ud s\\
		&+
		(\alpha-1) \int_{0}^{m} \int_{-\theta}^{0} e^{\lambda \tau} \beta_{Q}\left(s, \tau, Q_{*}(s)\right) p_{*}(s)\int_{0}^{s} w(y)\pi(\lambda, y)\ud y  \ud \tau \ud s\\
		&+
		\int_{0}^{m} \int_{-\theta}^{0} e^{\lambda \tau} \beta_{Q}\left(s, \tau, Q_{*}(s)\right) p_{*}(s) \ud \tau \ud s \int_{0}^{m} w(y)\pi(\lambda, y)\ud y -1.
	\end{aligned}
	$$
	Clearly condition \eqref{4.2} of Theorem \ref{4.8} is satisfied, thus we can restrict the characteristic equation $ \tilde{K}(\lambda)$ to $\lambda \in \mathbb{R}$.  Making use of $\beta_{Q}\left(s,\tau, Q_{*}\right) < 0$ and Eq. \eqref{2.2},  we obtain
	\begin{align*}
		\tilde{K}(0)=& 
		(\alpha-1) \int_{0}^{m} \int_{-\theta}^{0} \beta_{Q}\left(s, \tau, Q_{*}(s)\right) p_{*}(s)\int_{0}^{s} w(y)\pi(0, y)\ud y  \ud \tau \ud s\\
		&+
		\int_{0}^{m} \int_{-\theta}^{0}  \beta_{Q}\left(s, \tau, Q_{*}(s)\right) p_{*}(s) \ud \tau \ud s \int_{0}^{m} w(y)\pi(0, y)\ud y +\mathscr{R}(0,Q_{*}) -1\\
		=&
		\int_{0}^{m} \int_{-\theta}^{0} \beta_{Q}\left(s, \tau, Q_{*}(s)\right) p_{*}(s) \left(\alpha \int_{0}^{s} w(y)\pi(0, y)\ud y + \int_{s}^{m} w(y)\pi(0, y)\ud y \right) \ud \tau \ud s\\
		<&0.
	\end{align*}
	Moreover, we deduce that
	\begin{align*}
		\tilde{K}^{\prime}(\lambda)&= \int_{0}^{m}\int_{-\theta}^{0}e^{\lambda\tau}\beta(s,\tau,Q_{*}(s)) (\tau-\int_{0}^{s}\frac{1}{\gamma(a)} \d a) \pi(\lambda, s)\ud \tau \ud s\\
		&+
		(\alpha-1) \int_{0}^{m} \int_{-\theta}^{0} e^{\lambda \tau} \beta_{Q}\left(s, \tau, Q_{*}(s)\right) p_{*}(s)\int_{0}^{s} w(y)(\tau-\int_{0}^{s}\frac{1}{\gamma(a)} \ud a) \pi(\lambda, y)\ud y  \ud \tau \ud s\\
		&+
		\int_{0}^{m} \int_{-\theta}^{0} e^{\lambda \tau} \beta_{Q}\left(s, \tau, Q_{*}(s)\right) p_{*}(s) \int_{0}^{m} w(y) (\tau-\int_{0}^{s}\frac{1}{\gamma(a)} \d a) \pi(\lambda, y)\ud y \ud \tau \ud s.
	\end{align*}
	The tricky step needed here is to note that
	\begin{align*}
		&\int_{0}^{m} \int_{-\theta}^{0} e^{\lambda \tau} \beta_{Q}\left(s, \tau, Q_{*}(s)\right) p_{*}(s)\int_{0}^{s} w(y)(\tau-\int_{0}^{s}\frac{1}{\gamma(a)} \d a) \pi(\lambda, y)\ud y  \ud \tau \ud s\\
		= &\int_{0}^{m} \int_{-\theta}^{0} e^{\lambda \tau} w(s)
		\pi(\lambda, s)(\tau-\int_{0}^{s}\frac{1}{\gamma(a)} \ud a)  \int_{s}^{m} \beta_{Q}\left(y, \tau, Q_{*}(y)\right) p_{*}(y) \ud y  \ud \tau \ud s.
	\end{align*}
	By means of the positivity  condition \eqref{4.2}, we observe that
	\begin{align*}
		\begin{aligned}
			\tilde{K}^{\prime}(\lambda)=&\int_{0}^{m}\int_{-\theta}^{0}e^{\lambda\tau}\beta(s,\tau,Q_{*}(s)) (\tau-\int_{0}^{s}\frac{1}{\gamma(a)} \d a) \pi(\lambda, s)\ud \tau \ud s\\
			&+
			\alpha \int_{0}^{m} \int_{-\theta}^{0} e^{\lambda \tau} w(s)
			\pi(\lambda, s)(\tau-\int_{0}^{s}\frac{1}{\gamma(a)} \ud a)  \int_{s}^{m} \beta_{Q}\left(y, \tau, Q_{*}(y)\right) p_{*}(y) \ud y  \ud \tau \ud s\\
			&+
			\int_{0}^{m} \int_{-\theta}^{0} e^{\lambda \tau} w(s)\pi(\lambda, s)(\tau-\int_{0}^{s}\frac{1}{\gamma(a)} \ud a)  \int_{0}^{s} \beta_{Q}\left(y, \tau, Q_{*}(y)\right) p_{*}(y) \ud y  \ud \tau \ud s\\
			=&\int_{0}^{m} \int_{-\theta}^{0} e^{\lambda \tau} \pi(\lambda, s)(\tau-\int_{0}^{s} \frac{1}{\gamma(a)} \ud a)  \Big[ \beta\left(s, \tau, Q_{*}(s)\right) \\
			&+ w(s)\left(\int_{0}^{s} \beta_{Q}\left(y, \tau, Q_{*}(s)\right) p_{*}(y) \ud y+\alpha \int_{s}^{m} \beta_{Q}\left(y, \tau, Q_{*}(s)\right) p_{*}(y) \ud y\right)\Big] \ud \tau \ud s\\ \leq& 0.
		\end{aligned}
	\end{align*}
	As a result, for $\lambda \geq 0$, $\tilde{K}(\lambda)$ is monotone decreasing, and the stability result follows. Since $\tilde{K}(0)\geq 0$ by $\beta_{Q}\left(s,\tau, Q_{*}\right) \geq 0$ 
	and $\lim _{\lambda \rightarrow+\infty} \tilde{K}(\lambda)=-1$, the instability result follows  from the Intermediate Value Theorem.
\end{proof}

\section{Examples and simulations}

In this section we will present two examples to illustrate and underpin the linear stability results presented in Theorems \ref{th5.2} and \ref{th5.4}.

\begin{example}\label{ex6.0}(Stability of $p_{0}$)
	We set the model ingredients as follows: 
	$$\gamma\equiv 1,\, \mu\equiv 0.5,\, w\equiv 1,\, \alpha=0.5,\, \theta=1.5,\, m=8;$$
	$$
	\beta (s, \tau, Q(s,t+\tau))= \begin{cases} 0.5 e^{\tau}(0.7+sin^{2}(2s))(1-Q), & 0 \leq s \leq 8, \\ 0, & \text {otherwise.}\end{cases}
	$$
	\begin{figure}[H]
		\centering
		\begin{minipage}{0.49\linewidth}
			\centerline{\includegraphics[width=8.1cm]{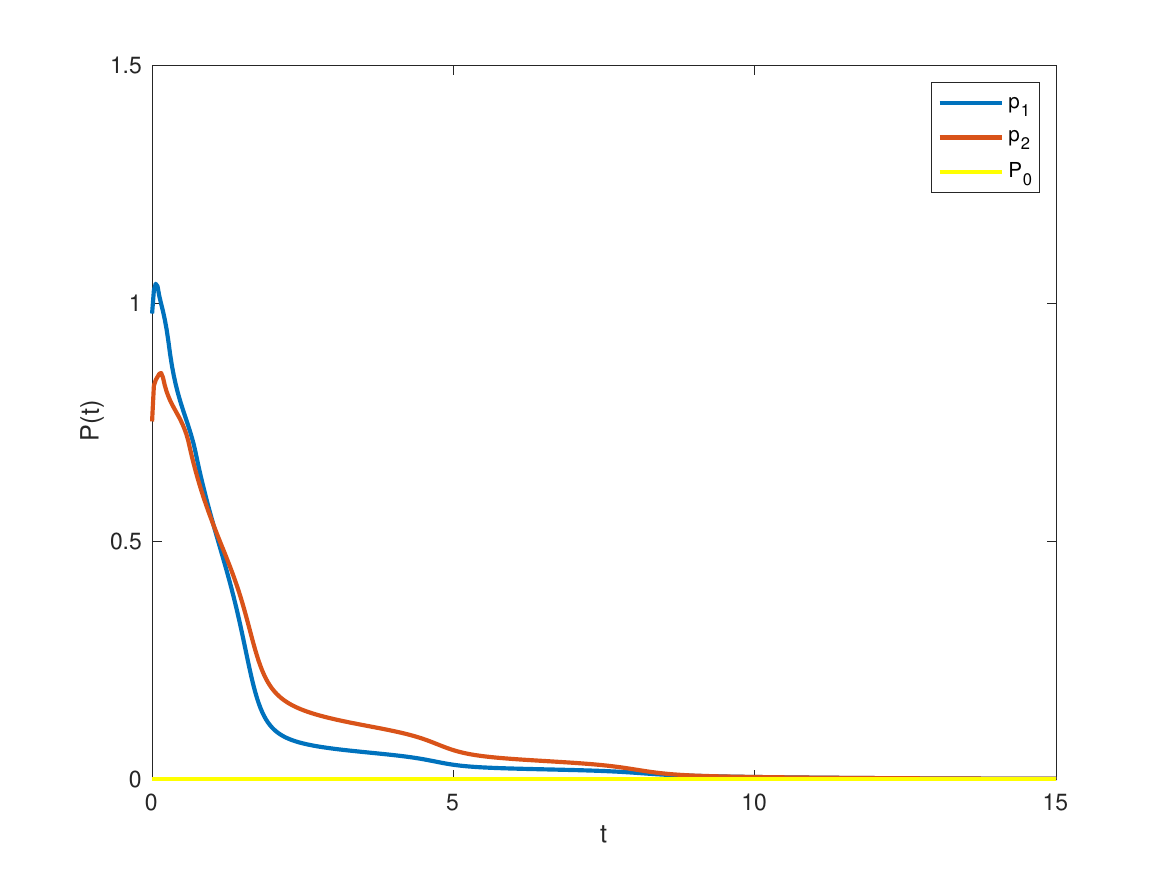}}
		\end{minipage}
		\hfill
		\begin{minipage}{0.49\linewidth}
			\centerline{\includegraphics[width=8.5cm]{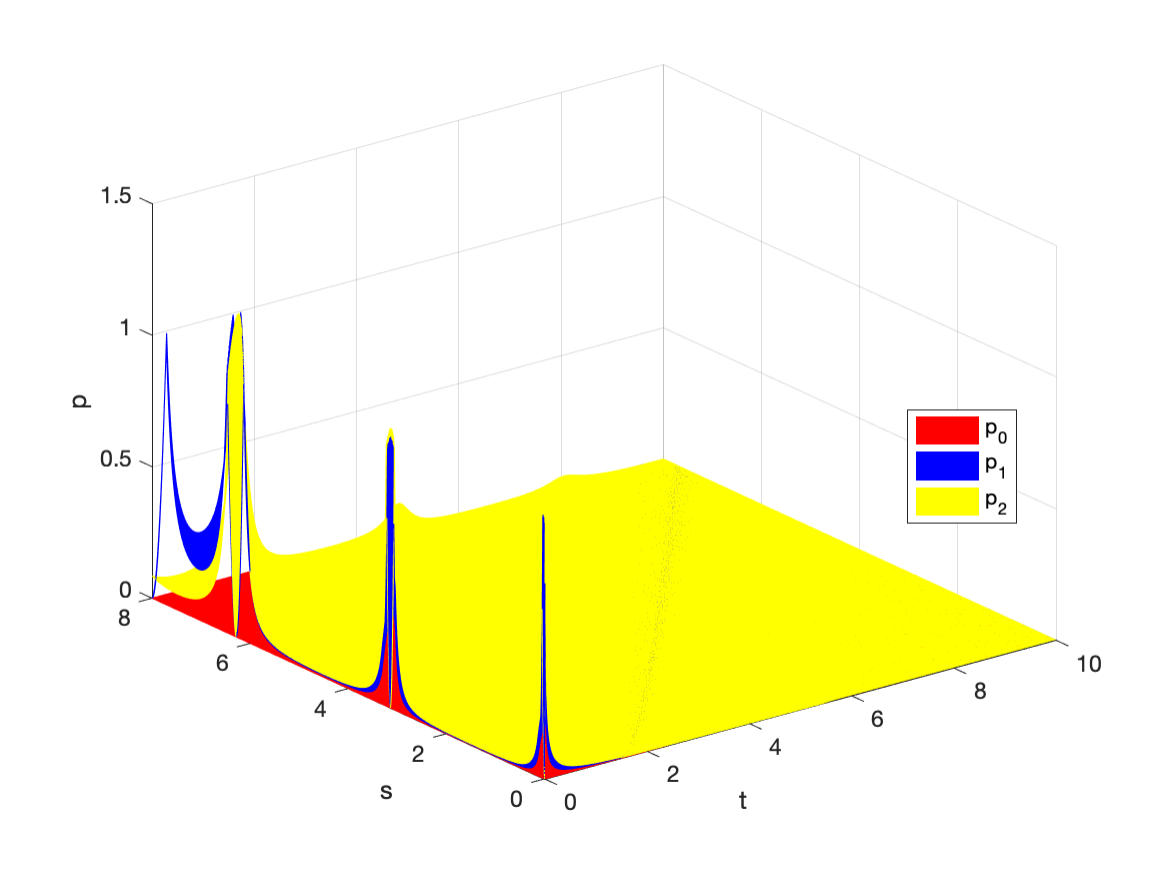}}
		\end{minipage}
		\hfill
		\caption{$\mathscr{R}(0,0)=0.9088$, $p_{0}$ represents the trivial stationary solution and $P(t)$ the total population size at time $t$; 
			the initial conditions corresponding to $p_{1}$ to $p_{2}$ are
			$u_{1}=12sin^2(s)(8-s)^2$ and  $u_{2}=3cos^2(s+\frac{\pi}{2})(10-s)^2.$
		}\label{stable0}
	\end{figure}
	We compute $\mathscr{R}(0,0)=0.9088<1$ using the inherent net  reproduction function \eqref{2.3}. 
	We can observe that as time increases, solutions approach the horizontal plane (trivial stationary solution), demonstrating the linear stability result in Theorem \ref{th5.2}, as shown in Fig.\ref{stable0}.

	When the fertility rate is changed to 
	$$
	\beta (s, \tau, Q(s,t+\tau))= \begin{cases} 0.55 e^{\tau}(1+cos^{2}(0.1s))(1-Q), & 0 \leq s \leq 8, \\ 0, & \text {otherwise,}\end{cases}
	$$
	we compute $\mathscr{R}(0,0)=1.6297>1$. As shown in Fig.\ref{unstable0}, the numerical results indicate that the solutions corresponding to the initial conditions $p_{1}$ and $p_{2}$ gradually move away from the horizontal plane, demonstrating the instability result presented in Theorem \ref{th5.2}.
	\begin{figure}[H]
		\centering
		\begin{minipage}{0.49\linewidth}
			\centerline{\includegraphics[width=8.1cm]{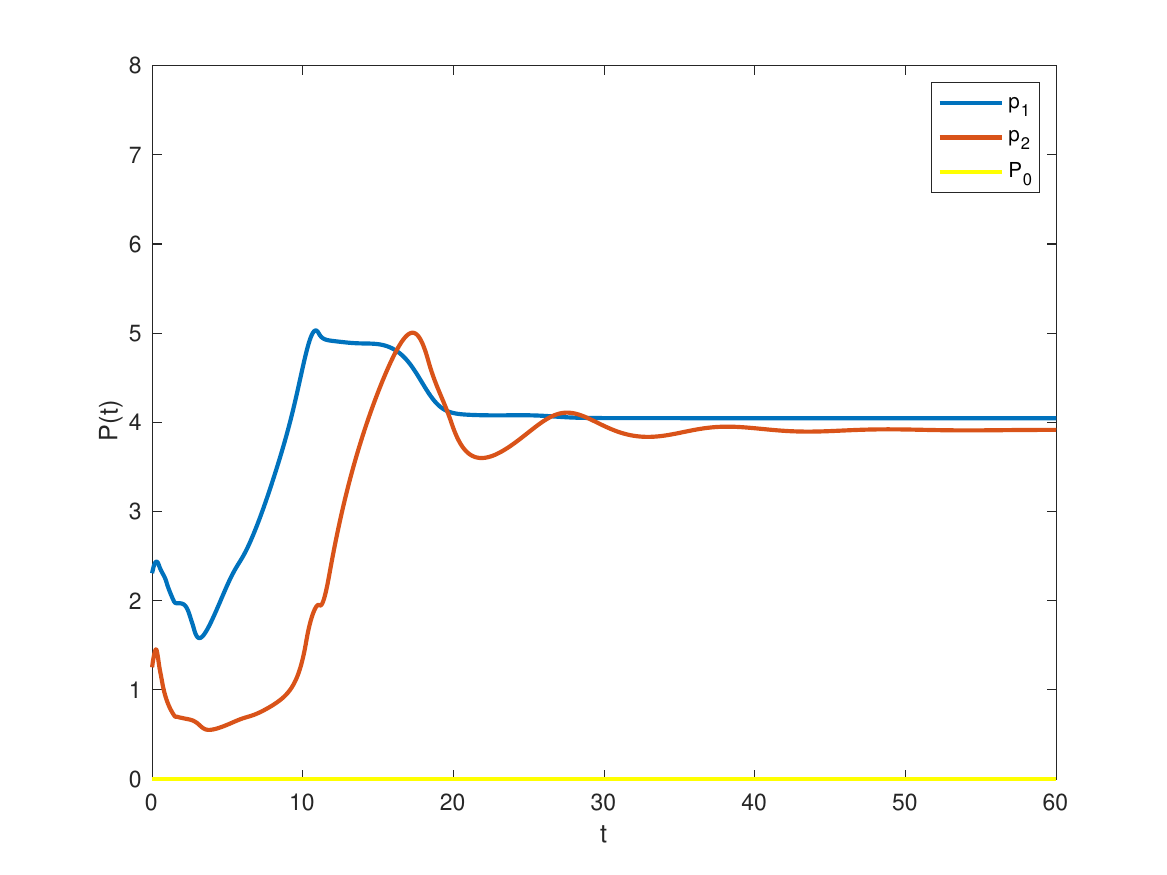}}
		\end{minipage}
		\hfill
		\begin{minipage}{0.49\linewidth}
			\centerline{\includegraphics[width=8.5cm]{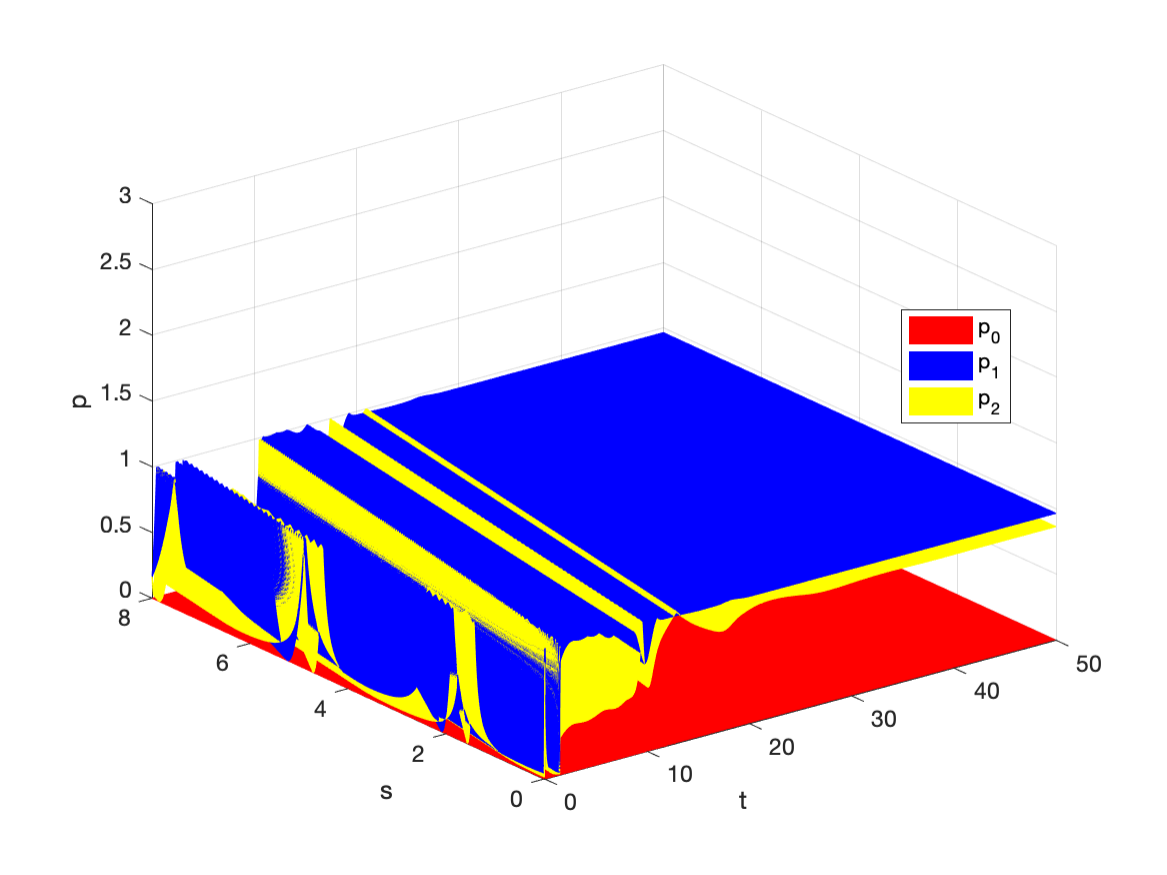}}
		\end{minipage}
		\hfill
		\caption{$ R(0,0)=1.6297$; the total population size $P(t)$ is plotted on the left; 
			the initial conditions corresponding to $p_{1}$ to $p_{2}$ are
			$u_{1}=0.3sin^2(s+\frac{\pi}{3})(10-s)^2; u_{2}=0.5sin^2(s+\frac{\pi}{2})(12-s)^2$.
		}\label{unstable0}
	\end{figure}
\end{example}

\begin{example}\label{ex6.1}(Stability of  $p_{*}$) Let us now consider the following set of model ingredients 
	$$\gamma\equiv 1,\, \mu\equiv 0.58,\,  w\equiv 1,\,  \alpha=0.6,\,  \theta=0.5;$$
	$$
	\beta (s, \tau, Q(s,t+\tau))= \begin{cases}  e^{\tau}(1+1.8s)(1-Q), & 0\leq Q\leq 1, \\ 0, & \text {otherwise.}\end{cases}
	$$
	It is not difficult to verify that both conditions \eqref{4.2} and $ \beta_{Q}\left(s,\tau, Q_{*}\right) < 0$ hold true for the current set of model ingredients. Here we take the initial conditions 
	$$u_{1}(s)=\frac{0.1}{0.1+10s^{3}}+0.028,\quad u_{2}(s)=\frac{0.1}{4+2s^{3}}+0.1; \quad\quad s\in[0,8].$$ 
	The numerical results show  that 
	total population sizes corresponding to the solutions $p_{1},p_{2}$  eventually converge to the total population size corresponding to the positive stationary solution $P_{*}$, which demonstrates the stability  result in Theorem \ref{th5.4}, as shown by  Fig.\ref{stable}.

	\begin{figure}[H]
		\centering
		\begin{minipage}{0.49\linewidth}
			\centerline{\includegraphics[width=8.1cm]{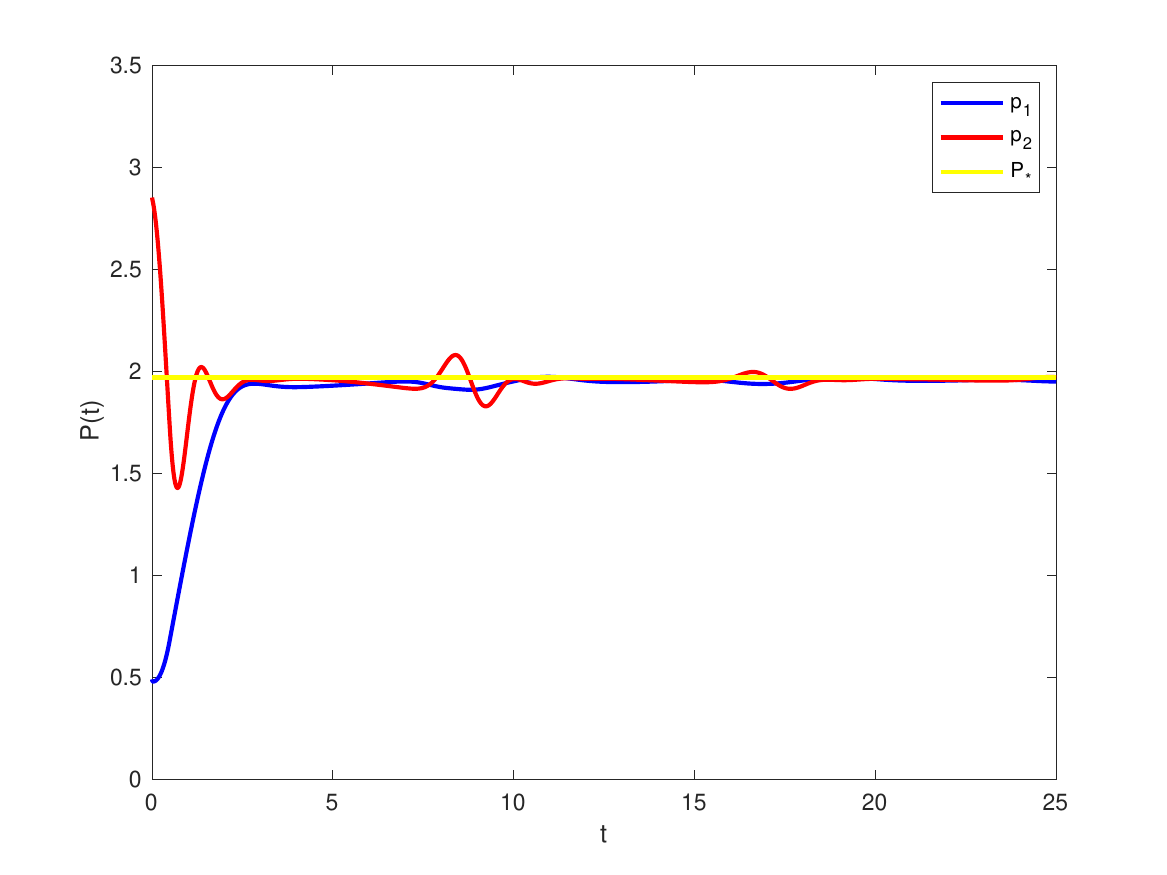}}
		\end{minipage}
		\hfill
		\begin{minipage}{0.49\linewidth}
			\centerline{\includegraphics[width=8.5cm]{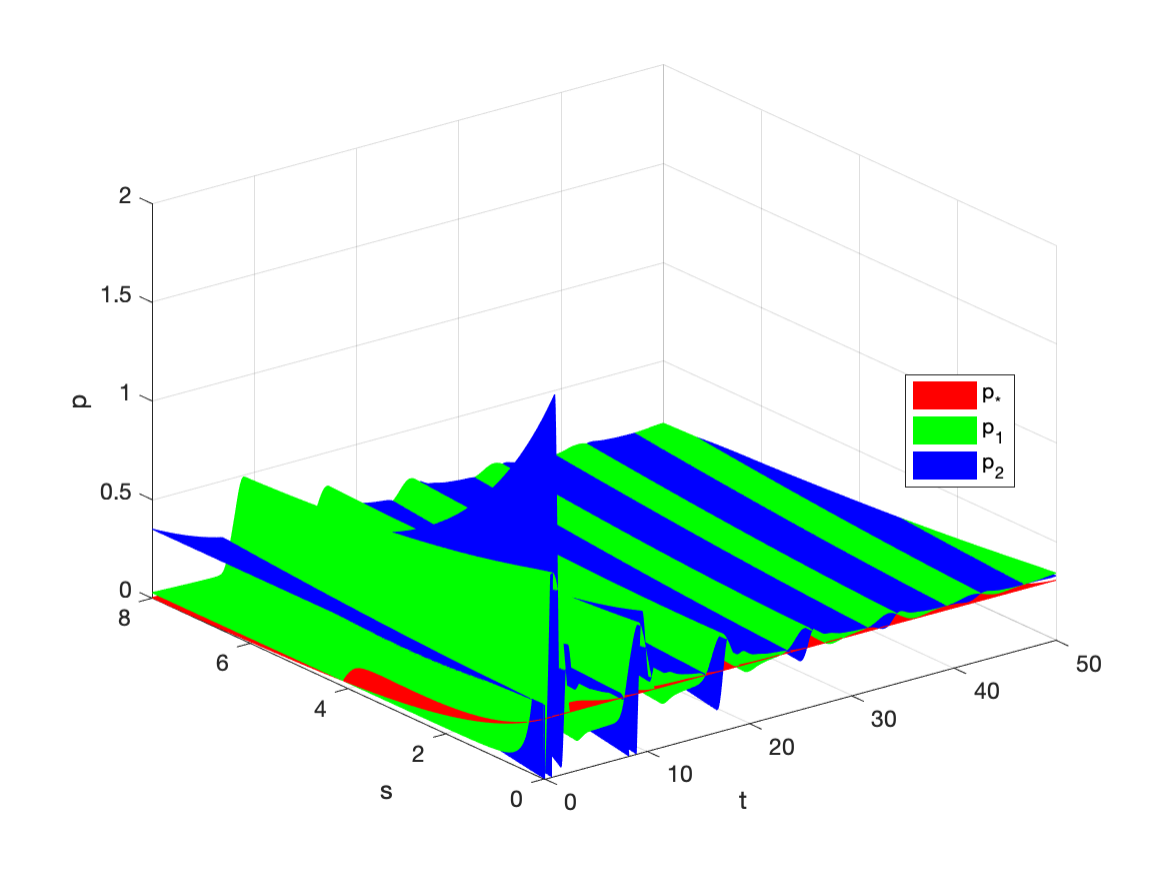}}
		\end{minipage}
		\hfill
		\caption{ $P(t)$ denotes the total population size at time $t$; $p_{*}$represents the stationary solution;
			the parameters $\gamma\equiv 1, \mu\equiv 0.58, w\equiv 1, \alpha\equiv 0.6, \theta=0.5, m=8$;
			the initial conditions corresponding to curves $p_{1}$ to $p_{2}$ are $u_{1}=\frac{0.1}{0.1+10s^{3}}+0.028$ and 
			$u_{2}=\frac{0.1}{4+2s^{3}}+0.1$.  On the left we can see the total population sizes plotted, while on the right the corresponding density distributions.  } \label{stable}
	\end{figure}
	Next we replace the fertility function with the following one
	$$
	\beta (s, \tau, Q(s,t+\tau))= \begin{cases} 0.5 e^{\tau}(1+0.1s)Q, &  Q \geq 0, \\ 0, & \text {otherwise.}\end{cases}
	$$
	It is obvious that conditions \eqref{4.2} and $ \beta_{Q}\left(s,\tau, Q_{*}\right) \geq 0$ of Theorem \ref{th5.4} are satisfied.
	The trajectories $p_{1}$ and $p_{2}$ are shown in Fig.\ref{unstable} with two different initial conditions. 
	This example demonstrates the instability result we obtained in Theorem \ref{th5.4}.

	\begin{figure}[H]
		\centering
		\begin{minipage}{0.49\linewidth}
			\centerline{\includegraphics[width=8.1cm]{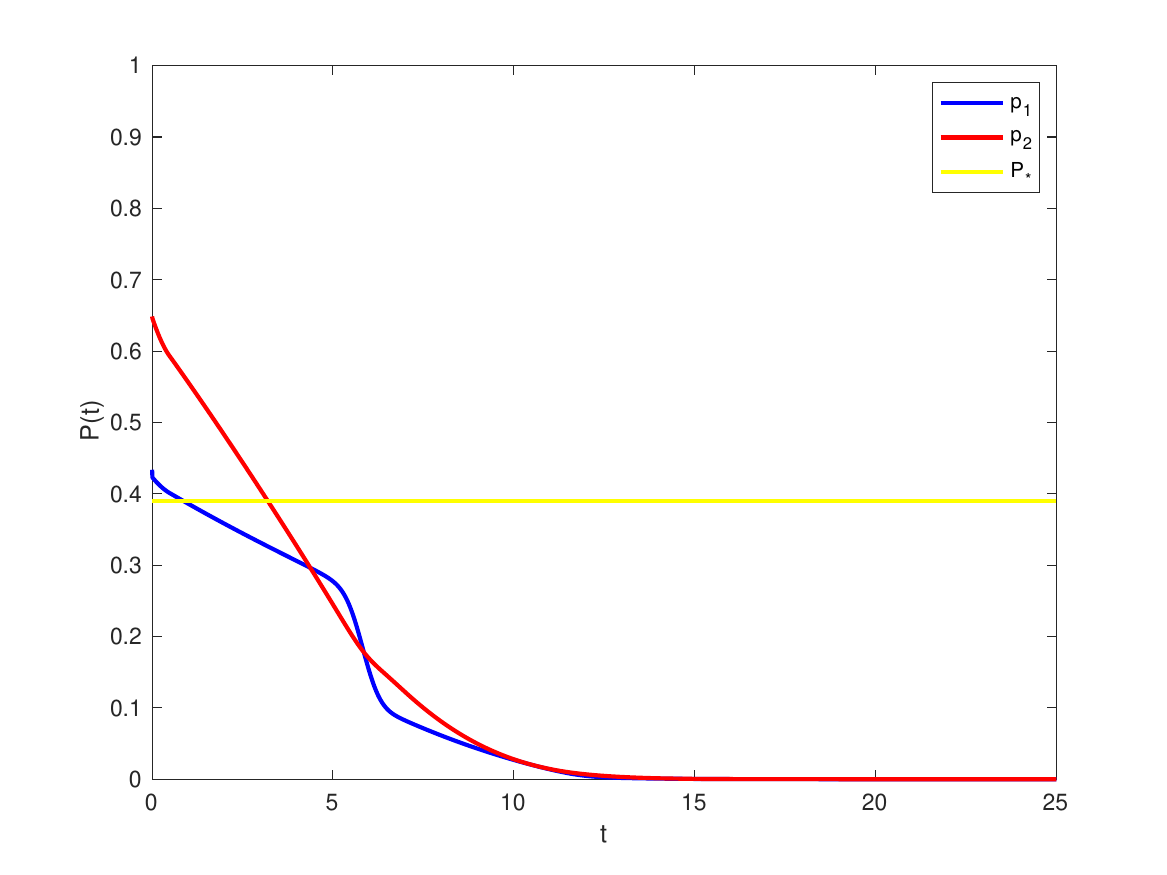}}
		\end{minipage}
		\hfill
		\begin{minipage}{0.49\linewidth}
			\centerline{\includegraphics[width=8.5cm]{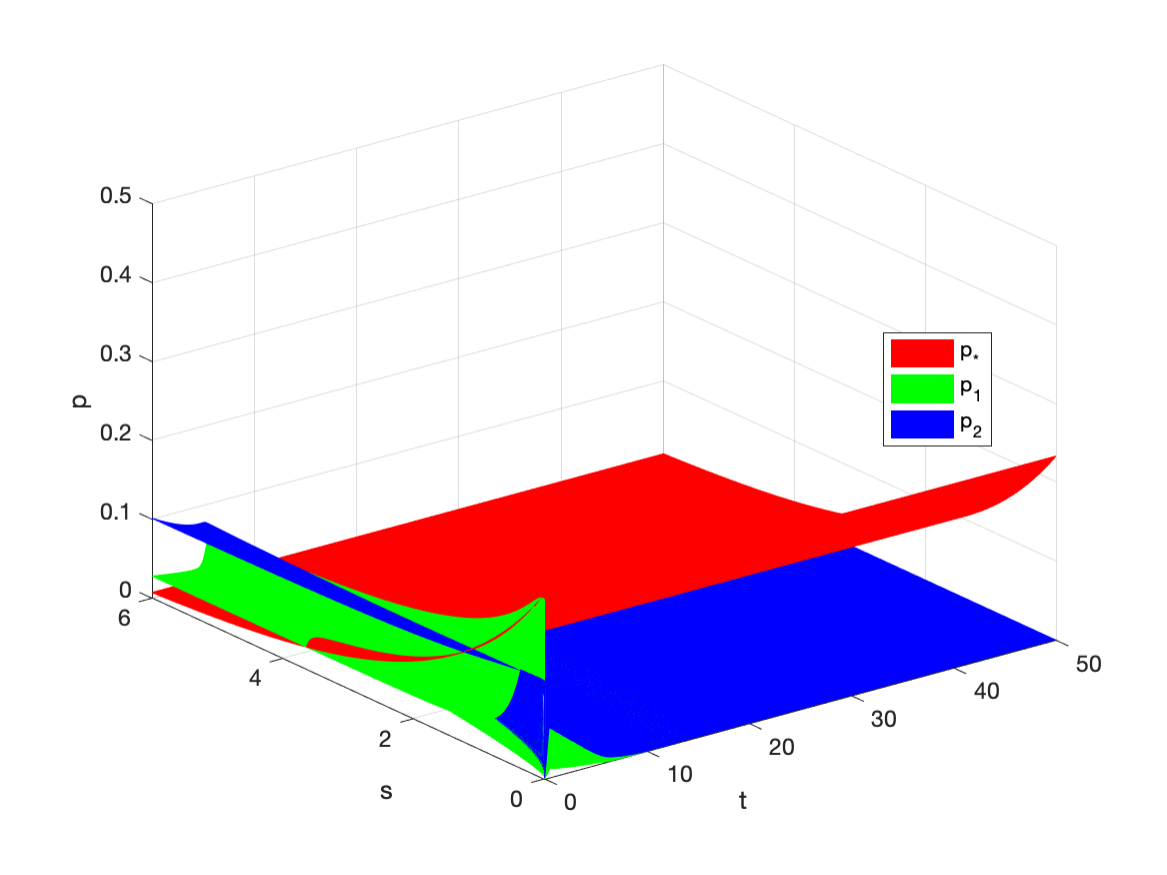}}
		\end{minipage}
		\hfill
		\caption{
			$P(t)$ denotes the total population size at time $t$; $p_{*}$represents the stationary solution;
			the parameters $\gamma\equiv 1, \mu\equiv 0.58, w\equiv 1, \alpha\equiv 0.6, \theta=0.5, m=8$;
			the initial conditions corresponding to curves $p_{1}$ to $p_{2}$ are $u_{1}=\frac{0.1}{0.1+10s^{3}}+0.028$ and 
			$u_{2}=\frac{0.1}{4+2s^{3}}+0.1$.  On the left we can see the total population sizes plotted, while on the right the corresponding density distributions. }\label{unstable}
	\end{figure}
	
\end{example}

\section{Conclusion}

In this work we have introduced and analysed a partial differential equation model intended to describe the dynamics of a hierarchical size-structured population. Our model incorporates two different types of nonlinearities: we assumed that individual growth and mortality are affected by scramble competition (which allows to model for example Allee effects); while recruitment of offspring is affected by contest competition via an infinite dimensional interaction variable related to a hierarchy in the population. Moreover, we incorporated delay in the recruitment (e.g. to account for maturation delay). We have formally linearised our model around a steady state and showed how to apply the theory of strongly continuous semigroups. In particular we studied the asymptotic behaviour of the governing semigroup by using spectral methods. In contrast to \cite{FH2008}, we were able to derive an explicit characteristic equation, which  characterises the point spectrum of the semigroup generator. This then allowed us to derive some stability/instability results, in particular using an appropriately defined net reproduction function.  The stability results we deduced were obtained by using a formal linearisation of the PDE model. A rigorous result often referred to as the Principle of Linearised Stability has not been established for the PDE model we studied here, therefore we presented examples and numerical simulations to underpin the formal stability results we established. 

Structured population models incorporating an infinite dimensional nonlinearity, e.g. due to a hierarchical structure in the population have been studied for long by many researchers. One of the earliest models describing a hierarchically age-structured population can be found in \cite{Cushing1994}. There is a major difference though between age-structured, i.e. semilinear, and size-structured, i.e. quasilinear models, such as the one we studied here. While natural age-structured PDE models tend to be well-posed on the biologically relevant state space of $L^1$; size-structured (quasilinear) models are not necessarily well-posed on $L^1$, in particular when the growth rate depends on the infinite dimensional nonlinearity (interaction variable) in a non-monotone fashion, see e.g. \cite{ASA2005-1,ASA2005}. In this case, in order to study existence of solutions, it is necessary to enlarge the state space and allow for measure valued solutions. The choice of the particular state space then becomes very important as demonstrated recently in \cite{AnnaJozsef}, in particular when trying to extend the theory of positive semigroups to such a setting.

\section*{Acknowledgments}
The authors are grateful to the editors and the anonymous referees for their valuable comments and suggestions which led to an improvement of our original manuscript. The Project was Supported by the Fundamental Research Funds for the Central Universities, China University of Geosciences (Wuhan) (NO.G1323523061).

\end{document}